\documentclass[10pt]{amsart}
\usepackage{amsmath,amssymb,amsfonts,amsthm,amsopn}
\usepackage{tgtermes}
\usepackage{latexsym,graphicx}
\usepackage{xcolor}
\usepackage{color, colortbl}
\usepackage{pb-diagram}
\usepackage[title]{appendix}
\usepackage{tikz-cd}
\usepackage{tikz}
\usepackage{mathtools}
\usepackage{enumerate}
\usepackage{hyperref}

\mathtoolsset{showonlyrefs}

\usepackage{multirow}

\newtheorem{theorem}{Theorem}[section]
\newtheorem{lemma}[theorem]{Lemma}
\newtheorem{corollary}[theorem]{Corollary}
\newtheorem{proposition}[theorem]{Proposition}
\newtheorem{definition}[theorem]{Definition}

\newtheorem{remark}[theorem]{Remark}
\theoremstyle{definition}
\newtheorem{example}[theorem]{Example}

\newcommand{\beqa}{\begin{eqnarray*}}
\newcommand{\eeqa}{\end{eqnarray*}}

\DeclareMathOperator*{\supp}{supp}

\DeclareMathOperator*{\id}{id}
\DeclareMathOperator*{\Sp}{Sp}
\DeclareMathOperator*{\Mp}{Mp}
\DeclareMathOperator*{\Sym}{Sym}
\DeclareMathOperator*{\diag}{diag}
\DeclareMathOperator*{\GL}{GL}

\newcommand{\field}[1]{\mathbb{#1}}

\newcommand{\bR}{\field{R}}        
\newcommand{\bN}{\field{N}}        
\newcommand{\bZ}{\field{Z}}        
\newcommand{\bC}{\field{C}}        
\newcommand{\bT}{\field{T}}        %

 \newcommand{\rd}{\mathbb{R}^n}       %
  \newcommand{\rdd}{\mathbb{R}^{2n}} 

\def\la{\lambda}

\def\cF{\mathcal{F}}              
\def\cS{\mathcal{S}}
\def\cD{\mathcal{D}}

\def\cB{\mathcal{B}}

\def\cK{\mathcal{K}}

\def\cA{\mathcal{A}}

\def\cI{\mathcal{I}}
\def\cC{\mathcal{C}}

\def\R{\right)}

\def\<{\left<}
\def\>{\right>}

\def\mv1{M_v^1}

\def\mn{(m,n)}
\def\mn'{(m',n')}

\newcommand{\abs}[1]{\lvert#1\rvert}
\newcommand{\norm}[1]{\lVert#1\rVert}

\def\R{\mathbb{R}}
\def\Ren{\mathbb{R}^n}

\def\f{\varphi}

\def\Sn2{S_{2}(L^{2}(\Ren))}
\def\S1{S_{1}(L^{2}(\Ren))}
\def\sig00{\sigma_{0,0}}

\def\la{\langle}
\def\ra{\rangle}

\newcommand{\Op}{\mathrm{Op}}
\newcommand{\w}{\mathrm{w}}

\begin{document}

\begin{abstract} 

In this work, we investigate the microlocal properties of the evolutions of Schr\"odinger equations using metaplectic Wigner distributions. So far, only restricted classes of metaplectic Wigner distributions, satisfying particular structural properties, have allowed the analysis of microlocal properties. We first extend the microlocal results to all metaplectic Wigner distributions, and examine these findings in the framework of Fourier integral operators  with quadratic phase. Finally, we analyze the phase space concentration of the (cross) Wigner distribution arising from the interaction of two states, with particular attention to interactions generated by certain Schr\"odinger evolutions. These contributions enable a more refined study of the so-called ghost frequencies.
\end{abstract}

\title[On the microlocal phase-space concentration of Wigner distributions associated with Schr\"odinger evolutions]{On the microlocal phase-space concentration of Wigner distributions associated with Schr\"odinger evolutions}

\author{Gianluca Giacchi}
\address{Università della Svizzera Italiana, IDSIA, Dipartimento di Informatica, Via La Santa 1, 6900 Lugano, Switzerland}
\email{gianluca.giacchi@usi.ch}
\author{Davide Tramontana}
\address{Università di Bologna, Dipartimento di matematica, Piazza di Porta San Donato 5, 40126 Bologna, Italy}
\email{davide.tramontana4@unibo.it}

\thanks{}
\maketitle
\tableofcontents                        

\section{Introduction}
Microlocal analysis and time-frequency analysis have long been regarded as two complementary perspectives on phase space analysis, offering distinct yet interconnected approaches and tools to the study of operators and singularities in phase space.
On the one hand, microlocal analysis is based on the notion of wave front set, introduced by H\"ormander with the aim of analyzing the microlocal $\cC^\infty$ and Sobolev singularities of distributions, and the directions in which these singularities occur, cf. \cite[Chapter 8]{hormander1} and \cite[Chapter 18]{hormander3}. Later, in \cite{hormander1495quadratic}, he introduced a different notion of wave front set, denoted by $WF_{\mathrm{iso}}$, tailored to detect Schwartz singularities. These singularities are called {\em isotropic} because the variables $x$ and $\xi$ are treated symmetrically. This concept has been extensively studied by many authors from both pseudodifferential and time-frequency analysis perspectives (see e.g. \cite{cappiellohanarmonic,cappielloFios,MPT2025,PravdaRodinoWahlberg,wahlberg2018propagation,MR4182467}). In the latter context, the notion is commonly referred to as the Gabor wave front set and it is given via the {\em short-time Fourier transform} (STFT) \cite{cordero2020time,grochenig2013foundations}, 
\begin{equation}\label{def.STFTintro}
	V_gf(x,\xi)=\int_{\rd}e^{-2\pi i\xi \cdot y}f(y)\overline{g(y-x)}dy, \qquad f,g\in L^2(\rd), \, (x,\xi)\in\rdd.
\end{equation}
Specifically,
\begin{definition}\label{def.WFG}
Let $f\in \cS'(\rd)$ and $g \in \mathcal{S}(\mathbb{R}^n)\setminus \lbrace 0 \rbrace$. We say that $0\neq X_0\notin WF_G(f)$, the {\em Gabor wave front set} of $f$, if there exists an open cone $\Gamma_{X_0}\subseteq\rdd$ containing $X_0$ such that $V_gf \in L_{v_s}^2(\Gamma_{X_0})$ for every $s\geq0$.
\end{definition}
In \cite{cordero2022wigner}, a further alternative definition of the wave front set is proposed basically by replacing the STFT with the {\em Wigner distribution} \cite{wigner1932quantum}, here normalized as
\begin{equation}\label{defWignerIntro}
	Wf(x,\xi)=\int_{\rd} e^{-2\pi i\xi\cdot y}f\left(x+\frac y 2\right)\overline{f\left(x-\frac y 2\right)}dy, \qquad f\in L^2(\rd).
\end{equation}
The reader may observe, in comparison with \eqref{def.STFTintro}, the absence of the Schwartz window function $g$. Precisely, the {\em Wigner wave front set} of $f \in L^2(\mathbb{R}^n)$ consists of all points $X_0\neq0$ in phase space where $\chi_{\Gamma_{X_0}}Wf$ does not decay faster than the inverse of any polynomial, where again $\Gamma_{X_0}$ is an open cone containing $X_0$, while $\chi_{\Gamma_{X_0}}$ is the characteristic function of $\Gamma_{X_0}$.

On the other hand, this qualitative approach of capturing singularities via wave front sets is reinterpreted by time-frequency analysis, where the STFT and the Wigner distribution are used as tools to quantify phase space concentration and how it changes through the action of evolution operators, pseudodifferential operators and, more generally, Fourier integral operators (FIOs), see \cite{MR4182445,MR4286055,MR3571901,cordero2010time,MR3391896,MR3392649,MR2294795,MR2356420,MR2208883} and the references therein. In the Weyl-Wigner formulation of non-commutative quantum mechanics, the Wigner distribution is modified to encode the geometry induced by a canonical deformation of the Heisenberg algebra  \cite{MR2432025,MR2718929}.
The Wigner distribution is a fundamental tool in both quantum mechanics and signal analysis, providing a powerful phase space representation of states and signals. Yet, in many concrete situations, ranging from quantum evolutions governed by non-standard symplectic forms to time-frequency transforms tailored to specific signal geometries, the classical Wigner distribution must be modified to properly capture the underlying structure \cite{MR2843220,PrataDias2018}. Metaplectic Wigner distributions, recently defined by Cordero and Rodino, provide a fruitful framework encompassing all the main generalizations of the Wigner distribution in the literature. Their contributions also show that the Wigner distribution can effectively be used to study the evolution operators of Schrödinger equations, the so-called Wigner analysis. Ghost frequencies, however, arise naturally from the non-linear nature of the Wigner distribution. Alongside the development of Wigner analysis, the propagation of singularities measured by means of the Wigner distribution has been studied, even in a metaplectic framework. However, only restricted classes of metaplectic Wigner distributions, whose expressions can be derived explicitly, have been considered so far.

Cordero and Rodino introduced {\em metaplectic Wigner distributions} as generalizations of the classical Wigner distribution encompassing many of the most significant and widely used time-frequency representations. To be definite, they use the metaplectic group $\Mp(2n,\bR)$, the double cover of the group $\Sp(2n,\bR)$ of $4n\times4n$ symplectic matrices, see Section \ref{subsec:MOps} below for the detailed definitions. The {\em (cross-)metaplectic Wigner distribution}, or (cross-)$\cA$-Wigner distribution, corresponding to $\hat\cA\in\Mp(2n,\bR)$, is then the time-frequency representation defined by
\begin{equation}
	W_\cA(f,g)=\hat\cA(f\otimes\bar g), \qquad f,g\in L^2(\rd).
\end{equation}
We write $W_\cA f:=W_\cA(f,f)$. Many time–frequency representations that fall within the framework of metaplectic Wigner distributions have been extensively studied from various perspectives over the past decade. We refer the reader to the works \cite{bayer2020linear,cordero2020linear,MR3471302,grochenig2024benedicks,grochenig2025characterization,grochenig2025more,Minh:2024aa,zhang2019uncertainty,zhang2021scaled}, among the aforementioned works and others, though the literature is considerably broader.
Our starting point is \cite{cordero2024wigner}, which ends with a brief discussion of $\cA$-{\em Wigner wave fronts}, focusing on metaplectic Wigner distributions $W_\cA$ satisfying certain structural properties, precisely {\em covariance} and {\em shift-invertibility}, we leave the details to Definitions \ref{def29} and \ref{def292} below, as they are secondary for the aim of this Introduction. 
\begin{definition}
	Let $W_\cA$ be covariant and shift-invertible. Let $f\in L^2(\rd)$. We say that $0\neq X_0\notin WF_\cA(f)$, the {\em $\cA$-Wigner wave front set} of $f$, if there exists an open cone $\Gamma_{X_0}\subseteq\rdd$ containing $X_0$ such that for every integer $s\geq0$,
	\begin{equation}
		\int_{\Gamma_{X_0}}\la X\ra^{2s}|W_\cA f(X)|^2dX<\infty.
	\end{equation}
\end{definition}
Under covariance and shift-invertibility assumptions, $W_\cA$ can be expressed explicitly, allowing to obtain the microlocal inclusion property for Weyl operators with symbols in the H\"ormander class of smooth functions that are bounded along with all of their derivatives, herein denoted as $S^0_{0,0}(\rdd)$, i.e.,
\begin{equation}
\Op^\mathrm{w}(a)f(x)=\int_{\rdd} e^{2\pi i(x-y)\xi}a\Big(\frac{x+y}{2},\xi\Big)f(y)dyd\xi,
\end{equation}
$a\in S^0_{0,0}(\rdd)$.
\begin{theorem}\cite[Theorem 7.6]{cordero2024wigner}\label{thmPartII1}
	Let $W_\cA$ be a shift-invertible and covariant metaplectic Wigner distribution. Then, for every $a\in S^0_{0,0}(\rdd)$ and every $f\in L^2(\rd)$,
	\begin{equation}
		WF_\cA(\Op^\w(a)f)\subseteq WF_\cA(f).
	\end{equation}
\end{theorem}

The crucial tool for this analysis was the so-called {\em Wigner kernel}. Namely, for a given linear and continuous operator $T:\cS(\rd)\to\cS'(\rd)$, Wigner considered another operator $K:\cS(\rdd)\to\cS'(\rdd)$ satisfying the intertwining relation
\begin{equation}\label{defdiK}
	W(Tf)=K(Wf), \qquad f\in L^2(\rd).
\end{equation}
The formal definition of $K$, investigated in \cite{cordero2024understanding,cordero2023wigner}, requires the use of the {\em cross-Wigner distribution}, 
\begin{equation}\label{defWigner-gg}
	W(f,g)(x,\xi)=\int_{\rd} e^{-2\pi i\xi\cdot y}f\left(x+\frac y 2 \right)\overline{g\left(x-\frac y 2 \right)}dy, \qquad x,\xi\in\rd,\, f,g\in L^2(\rd),
\end{equation}
If $f=g$, we retrieve \eqref{defWignerIntro}. For $T$ as above, the authors of the aforementioned works have proven the existence of a unique tempered distribution $k\in\cS'(\bR^{4n})$, called the {\em Wigner kernel} of $T$, which satisfies
\begin{equation}
	W(Tf,Tg)(X)=\int_{\rdd}k(X,Y)W(f,g)(Y)dY, \qquad f,g\in \cS(\rd),
\end{equation}
with distributional interpretation of the integral. As a byproduct, they obtain the existence of $K$ in \eqref{defdiK}, as the operator with kernel $k$. 

The analysis through Wigner kernels, called {\em Wigner analysis}, turned out to be very powerful in the description of pseudodifferential operators ($\Psi$DOs) and, more generally, of FIOs with quadratic phase and symbols in $S^0_{0,0}(\rdd)$. For a more systematic treatment of $\cA$-wave front sets and to obtain results like Theorem \ref{thmPartII1} in a full-generality, we need to extend Wigner analysis in a metaplectic framwork.
\begin{definition}\label{defkAintro}
	Fix a metaplectic Wigner distribution $W_\cA$. The {\em $\cA$-Wigner kernel} of a linear and continuous operator $T:\cS(\rd)\to\cS'(\rd)$ is the tempered distribution $k_\cA\in\cS'(\bR^{4n})$ such that 
	\begin{equation}
		W_\cA(Tf,Tg)(X)=\int_{\rdd}k_\cA(X,Y)W_\cA(f,g)(Y)dY, \qquad f,g\in\cS(\rd), \qquad X\in\rdd.
	\end{equation}
\end{definition}
The existence of the $\cA$-Wigner kernel of Definition \ref{defkAintro} implies that there exists an operator $K_\cA:\cS(\rdd)\to\cS'(\rdd)$ so that
\begin{equation}\label{defKAintro}
		W_\cA(Tf,Tg)=K_\cA W_\cA(f,g), \qquad f\in\cS(\rd),
	\end{equation}
	as detailed in Section \ref{sec:WKers}. 
The first main results of this work concerns how the phase space concentration, measured by means of an arbitrary $\cA$-Wigner distribution, changes under the action of Weyl operators with symbols in $S_{0,0}^0(\rdd)$. Our definition of metaplectic wave front is as follows.
\begin{definition}\label{defWFtau-intro}
Let $f \in L^2(\mathbb{R}^n)$ and $\hat\cA \in \Mp(2n,\mathbb{R})$. We say that $0 \neq X_0 \notin WF_\cA (f)$, \emph{the $\cA$-Wigner wave front set} of $f$, 
if there exists an open cone $\Gamma_{X_0}\subseteq \dot{\R}^{2n}$, containing $X_0$ such that 
\begin{equation}\label{condWFA2}
W_\cA f \in \bigcap_{s \geq 0} L^2_{v_s}(\Gamma_{X_0}).
\end{equation}
\end{definition}
We obtain the following generalization of Theorem \ref{thmPartII1}.
\begin{theorem}\label{theo.microlocalitycA}
	Let $a\in S^0_{0,0}(\rdd)$ and $W_\cA$ be a metaplectic Wigner distribution. Then, for every $f\in L^2(\rd)$,
	\begin{equation}\label{microlocalitycA}
		WF_\cA(\Op^\w(a)f)\subseteq WF_\cA(f).
	\end{equation}
\end{theorem}
We also examine the behavior of the phase space concentration under the action of a FIO with quadratic phase $\Phi$,
\begin{equation}
	\cK f(x)=\int_{\rd} e^{2\pi i\Phi(x,\xi)}a(x,\xi)\hat f(\xi)d\xi,
\end{equation}
see Definition \ref{def.quadraticfio} for the details. However, a finer analysis of these operators requires to restrict to the Wigner wave front set $WF_\w$ already considered in \cite{cordero2022wigner}. 
\begin{theorem}\label{theo.microFIO}
Let $\cK$ be a quadratic FIO {associated} with canonical mapping $S$ and symbol $a \in S_{0,0}^0(\mathbb{R}^{2n})$, and let $f \in L^2(\mathbb{R}^n)$. Then, 
\begin{equation}\label{microlocalityFIO}
WF_\mathrm{w}(\mathcal{K}f) \subseteq S(WF_\mathrm{w}(f)).
\end{equation}
\end{theorem}
Most of the time, results such as Theorem \ref{theo.microFIO} can also be established for all metaplectic Wigner distributions $W_\cA$. However, the explicit expression of $W_\cA$ is typically required to get \eqref{microlocalityFIO}, making it practically impossible to obtain a statement for general wavefronts $WF_\cA$. For this reason, when it comes to FIOs, we choose to remain in the classical setting of the Wigner distribution.

Next, we study the propagation of phase space concentration of the Wigner distribution associated with two states, in terms of the following wave front sets.
\begin{definition}\label{def.crossWF}
Let $f,g \in L^2(\mathbb{R}^n)$. 
We say that $0\neq X_0 \notin WF^s_\cA(f,g)$
{\em the (cross) $\cA$-Wigner wave front set of $f$ and $g$ of order $s \geq 0$},  
if there exists an open cone $\Gamma_{X_0}\subseteq \dot{\R}^{2n}$ containing $X_0$ such that  $W_\cA(f,g) \in L^2_{v_s}(\Gamma_{X_0})$. In addition, we say that $0\neq X_0 \notin WF_\cA (f,g)$, \emph{the (cross) $\cA$-Wigner wave front set of $f$ and $g$},  
if there exists an open cone $\Gamma_{X_0}\subseteq \dot{\R}^{2n}$ containing $X_0$ such that $W_\cA(f,g) \in L^2_{v_s}(\Gamma_{X_0})$, for every $s\in \mathbb{R}$.
\end{definition}
This notion provides a natural tool to describe correlations between states, how they mutually affect their singularities and eventually have insights on their propagation in phase space. Indeed, the interaction between two states produces the so-called {\em ghost frequencies}, fast oscillations that show up when computing the Wigner distribution of a signal, due to the absence of the mitigating effect of the window function present instead in Definition \ref{def.WFG}, see \cite{cordero2024wigner,cordero2025wigner,cordero2022wigner} and the quadratic nature of the Wigner distribution. For the sake of clarity, consider a limit case where $f=1$ and $g=\delta$, the point-mass distribution. Then, 
\begin{equation}
	W(1+\delta)=1\otimes\delta+\delta\otimes 1+2^{d+1}\cos(4\pi x\xi).
\end{equation}
The third addendum is the real part of $W(f,g)$ and produces oscillations that neither $f$ nor $g$ themselves exhibit, precisely the ghost frequencies \cite{cordero2025wigner}. We are therefore naturally led to study the mutual interaction, from a microlocal perspective, of $u_1$ and $u_2$ Schr\"odinger evolutions of $u_{0,1}, u_{0,2} \in L^2(\R^n)$ determined by
\begin{equation}\label{eq.cauchyj}
\begin{cases}
\frac{1}{2\pi}i\partial_t u_j=H_ju_j & \text{on $\bR \times \R^n$},\\
u_j(0,\cdot)=u_{0,j} & \text{on $\rd$}.
\end{cases}
\end{equation}
Here, the Hamiltonians are
\[
H_1 = \mathrm{Op}^{\mathrm{w}}(a_1) + \mathrm{Op}^{\mathrm{w}}(\sigma_1), \quad H_2 = \mathrm{Op}^{\mathrm{w}}(a_2) + \mathrm{Op}^{\mathrm{w}}(\sigma_2),
\]
where $a_1, a_2\in\cS'(\rdd)$ are real quadratic forms and $\sigma_1, \sigma_2\in S^0_{0,0}(\rdd)$ are bounded perturbations, see also Section \ref{sec.cross} for the details. For $j=1,2$, the solutions to \eqref{eq.cauchyj} can be expressed as {\em generalized metaplectic operators}
\begin{equation}\label{introGMOps}
	e^{2\pi itH_j}u_{0,j}=\widehat{S_{j,t}}\Op^\w(b_{j,t})u_{0,j},
\end{equation}
where, $t\in\bR\mapsto S_{j,t}\in\Sp(n,\bR)$ describes the solution of the classical equations of motion with Hamiltonian $a_j(x,\xi)$ in phase-space, and the symbols $b_{j,t}\in S^0_{0,0}(\rdd)$, see \cite{cordero2014generalized}.
In the case where $a_1=a_2$ the propagation of interferences between $u_1$ and $u_2$ is governed by the following result.
\begin{theorem}\label{theo.crossa1equala2}
Let $u_{0,1},u_{0,2} \in L^2(\mathbb{R}^n)$ and $t\in\bR$. If $a_1=a_2$ and $u_j(t,\cdot)=e^{itH_j}u_{0,j}$ denotes the solution of \eqref{eq.cauchyj} at time $t$ ($j=1,2$), then 
\[
WF_\w(u_1(t,\cdot),u_2(t,\cdot)) \subseteq S_t \, WF_\w(u_{0,1},u_{0,2}),
\]  
where $S_t=S_{1,t}=S_{2,t}$, according to \eqref{introGMOps}.
\end{theorem}
Furthermore, when $a_1\neq a_2$ we may study the interaction in terms of the corresponding metaplectic operators as shown in the next Theorem. 
\begin{theorem}\label{theo.crossa1neqa2}
Let $t\in\bR$. Then, if $\cA \in \Sp(2n,\mathbb{R})$ and if $a_1\neq a_2$, we have 
\[
WF_\cA(u_1(t,\cdot),u_2(t,\cdot)) \subseteq WF_{\cC_t}(u_{0,1},u_{0,2}),
\]
where $\cC_t=\cA\cS_t$, with $\cS_t=S_{1,t} \otimes \overline{S_{2,t}}$, see \eqref{tensmates} and \eqref{conjmates}.
\end{theorem}
 
{\bf Outline.} In Section \ref{sec.pre} we recall the main facts of symplectic and pseudodifferential theory that we need throughout the work. In Section \ref{sec:WKers} we study the $\cA$-Wigner Kernels used in Section \ref{sec.Awigmic} to investigate how the microlocal phase space concentration of the ($\cA$-)Wigner distribution changes under the action of Weyl operators. In Section \ref{sec:WMFIOs} we study the microlocal properties of FIOs with quadratic phases and bounded symbols. In Section \ref{sec.cross}, we address the problem of examining the behavior of the $\cA$-Wigner distributions corresponding to the interaction of two states. The conclusion is withdrawn in Section \ref{sec:conclusion}, where we also compare the classical framework of Gabor wavefronts with Wigner wave fronts.

We point out that theory stated in this work for operators having symbols in $S^0_{0,0}(\rdd)$ can be rephrased verbatim to far wider classes of symbols, the so-called {\em weighted Sj\"ostrand classes} \cite{MR1266757}. The intersection of these spaces coincides with $S^0_{0,0}(\rdd)$, as it is detailed in \cite[Lemma 2.3]{BastianoniDecay}, and eventually our results can be derived as particular instances of this more general framework. Our restriction to $S^0_{0,0}(\rdd)$ is therefore made purely for the sake of simplicity.

\section{Prerequisites}\label{sec.pre}
We denote by $xy=x\cdot y=\langle x, y \rangle_{\rd}$ the standard inner product in $\mathbb{R}^{n}$, whereas $B_R=\{x\in\mathbb{R}^{n}:|x|<R\}$. If $A\subseteq \rd$, then $A^c$ denotes its complement in $\rd$. We denote by $I_n$ the $n\times n$ identity matrix and by $0_n$ the $n\times n$ null matrix. If $Q\in\bR^{n\times n}$, we denote by $Q^\top$ its transpose. The set of real $n\times n$ symmetric matrices is denoted by $\Sym(n,\bR)$. The group of real $n\times n$ invertible matrices is denoted by $\GL(n,\bR)$. By $\langle \cdot,\cdot \rangle_{\mathbb{C}^{2n}}$ we mean the non-Hermitian form on $\mathbb{C}^{2n}$ defined by $\langle X,Y \rangle_{\mathbb{C}^{2n}}=\sum_iX_iY_i$ for $X,Y \in \mathbb{C}^{2n}$.
We denote by $X=(x,\xi) \in \mathbb{R}^{2n}$ the global variable of the phase space and we write $\dot{\mathbb{R}}^{2n}=\mathbb{R}^{2n}\setminus \lbrace 0 \rbrace$.
With $\mathbb{N}_0$ we mean $\mathbb{N} \cup \lbrace 0 \rbrace$. 
We say that a set $\Gamma\subseteq \dot{\mathbb{R}}^{2n}$ is \textit{conic} if it is invariant under multiplication by positive real numbers. Hence a cone (i.e. a conic set) $\Gamma \subseteq \dot{\mathbb{R}}^{2n}$ is uniquely determined by its intersection with the unit sphere $\mathbb{S}^{2n-1}$. We write $A \approx B$ if there exists a universal constant $C>0$ such that $C^{-1}B \leq A \leq CB$ and $A\lesssim B$ if $A \leq CB$. By abuse,
if $f\in\cS'(\rd)$ and $E\in\GL(n,\bR)$, we write $f(Ex)$ to indicate the pullback of $f$ through the diffeomorphism induced by $E$. 

The Schwartz class of rapidly decreasing functions is denoted by $\cS(\mathbb{R}^{n})$. Its topological dual is the space of tempered distributions. The notation $\langle \cdot,\cdot\rangle$ stands for both the sesquilinear inner product in $L^2(\mathbb{R}^{n})$, $\la f,g\ra =\int_{\mathbb{R}^{n}}f(x)\overline{g(x)}dx$, $f,g\in L^2(\mathbb{R}^{n})$, and for its unique extension to a duality pairing on $\cS'\times\cS$ (antilinear in the second component), i.e., $\la f,\varphi\ra=f(\bar\varphi)$, $f\in \cS'(\mathbb{R}^{n})$, $\varphi\in\cS(\mathbb{R}^{n})$. Moreover, by $\|\cdot \|$ we mean the $L^2$-norm. The flip operator is $\cI g(t)=g(-t)$.

We write $v_s(x)=\la x\ra^s=(1+|x|^2)^{s/2}$, the Japanese brackets. A complex-valued measurable function $f$ on $\rd$ belongs to $L^2_{v_s}(\rd)$ if its weighted Lebesgue norm
\begin{equation}
	\norm{f}_{L^2_{v_s}}=\left(\int_{\rd}|f(x)|^2v_s(x)^2dx\right)^{1/2}
\end{equation} 
is finite. If $\Gamma\subseteq\rd$, we write $f\in L^2_{v_s}(\Gamma)$ if $\chi_{\Gamma}f\in L^2_{v_s}(\rd)$, where $\chi_\Gamma$ is the characteristic function of $\Gamma$.

In what follows, for $x,\xi \in \mathbb{R}^n$ we denote by $T_x$ and $M_\xi$ the translation and the modulation operators respectively, defined as 
\[
T_xf(y)=f(y-x), \quad M_\xi f(y)=e^{2\pi i \xi y}f(y), \quad f \in L^2(\mathbb{R}^n).
\]
Then, for $(x,\xi) \in \mathbb{R}^{2n}$, the \emph{time-frequency shift} is the operator $\pi(x,\xi)$
defined by
\[
\pi(x,\xi)f=M_\xi T_x f, \quad f \in L^2(\mathbb{R}^n).
\]
For the Fourier transform, we use the normalization
\begin{equation}\label{FTdef}
	\hat f(\xi)=\int_{\rd}e^{-2\pi ix \xi}f(x)dx, \qquad f\in L^1(\rd), \; \xi\in\rd.
\end{equation}

\subsection{Wigner distribution}\cite{MR3643624,folland1989harmonic,grochenig2013foundations,lernermetrics}

Recall the definition of the cross-Wigner distribution given in \eqref{defWigner-gg} for $f,g\in L^2(\rd)$. Its main continuity properties are summarized by the following result.
\begin{proposition}
The operator $W$ maps $\mathcal{S}(\R^n) \times \mathcal{S}(\R^n)$ into $\mathcal{S}(\R^{2n})$ and extends to a map $\mathcal{S}'(\R^n) \times \mathcal{S}'(\R^n)$ into $\mathcal{S}'(\R^{2n})$. Moreover, $W$ maps $L^2(\R^n) \times L^2(\R^n)$ into $L^2(\R^{2n})\cap \mathcal{C}(\R^{2n})$ and 
\[
\langle W(f_1,g_1), W(f_2,g_2) \rangle =\langle f_1,f_2 \rangle \overline{\langle g_1,g_2 \rangle}, \quad f_1,f_2,g_1,g_2 \in L^2(\R^n).
\]
\end{proposition}

\subsection{Symplectic matrices}\cite{folland1989harmonic,grochenig2013foundations}
	Let
	\begin{equation}\label{defJ}
		J=\begin{pmatrix}
			0_n & I_n\\
			-I_n & 0_n
		\end{pmatrix}
	\end{equation}
	be the matrix of the standard symplectic form of $\mathbb{R}^{2n}$. A matrix $S\in\bR^{2n\times2n}$ is \textit{symplectic} if $S^\top JS=J$. We denote by $\Sp(n,\bR)$ the group of $2n\times 2n$ symplectic matrices. If
	\begin{equation}\label{blockS}
		S=\begin{pmatrix}
			A & B\\
			C & D
		\end{pmatrix},
	\end{equation}
	with $A,B,C,D\in\bR^{n\times n}$, then $S\in \Sp(n,\bR)$ if and only if
	\begin{align}
		\label{Sp1}
		&A^\top C=C^\top A,\\
		\label{Sp2}
		&B^\top D=D^\top B,\\
		\label{Sp3}
		&A^\top D-C^\top B=I_n.
	\end{align}
	\begin{proposition}
		The symplectic group $\Sp(n,\bR)$ is generated by matrices in the form
		\begin{equation}\label{defDEVQ}
			\cD_E=\begin{pmatrix}
			E^{-1} & 0_n\\
			0_n & E^\top
			\end{pmatrix},
			\quad 
			V_Q=\begin{pmatrix}
			I_n & 0_n\\
			Q & I_n
			\end{pmatrix}, \quad E\in \GL(n,\bR), \ Q\in\Sym(n,\bR)
		\end{equation}
		and by the matrix $J$ defined in \eqref{defJ}.
	\end{proposition}
	
	In this work, we shall also consider $4n\times4n$ symplectic matrices. To distinguish these matrices from $2n\times 2n$ symplectic matrices, we denote the former with calligraphic capital letters (e.g., $\cA$), and the latter with capital letters (e.g. $S$). 

\subsection{Metaplectic operators}\label{subsec:MOps}
	Consider \textit{the Schr\"odinger representation} of the Heisenberg group:
	\begin{equation}\label{eq19}
		\rho(x,\xi;\tau)g(y)=e^{2\pi i\tau}e^{-i\pi x\xi}e^{2\pi i\xi y}g(y-x), \qquad x,\xi\in\mathbb{R}^{n}, \; \tau\in\bR, \; g\in L^2(\mathbb{R}^{n}).
	\end{equation}
	For every symplectic matrix $S\in \Sp(n,\bR)$ there exists an operator $\hat S$ unitary on $L^2(\mathbb{R}^{n})$ such that
	\begin{equation}\label{intertMp}
		\hat S\rho(x,\xi;\tau)\hat S^{-1}=\rho(S(x,\xi);\tau), \qquad x,\xi\in\mathbb{R}^{n}, \; \tau\in\bR.
	\end{equation}
	Any such operator $\hat S$ is called a {\em metaplectic operator}. A careful analysis of \eqref{intertMp} reveals that if it is satisfied by $\hat S$, then it is also satisfied by $c\hat S$ for every $c\in\bT=\{z\in\bC:|z|=1\}$ ({\em phase factor}). The group $\{\hat S:S\in\Sp(n,\bR)\}$ has a subgroup consisting of exactly two metaplectic operators for each symplectic matrix, the so-called {\em metaplectic group}, denoted by $\Mp(n,\bR)$. The projection $\pi^{Mp}:\hat S\in \Mp(n,\bR)\mapsto S\in\Sp(n,\bR)$ is a group homomorphism with kernel $\{\pm \id_{L^2}\}$. For the theory of metaplectic operators we refer to \cite{cordero2020time}, but the interested reader may also refer to \cite{de2021quantum,folland1989harmonic,grochenig2013foundations}. 
	
	In synthesis, a symplectic matrix $S$ defines a metaplectic operator up to a phase factor. For the purposes of this work, this constant is irrelevant and, therefore, we omit its specific choice in the discussion whenever it does not cause confusion.
	
	\begin{proposition}\label{propContMetap}
		The metaplectic group $\Mp(n,\bR)$ is generated by the operators in the form
		\begin{align}
			\label{FTMp}
			&\cF f(\xi)=i^{-n/2} \hat f(\xi)= i^{-n/2}\lim_{R\to+\infty}\int_{B_R}e^{-2\pi i\xi x}f(x)dx,\\
			\label{TEMp}
			&\mathfrak{T}_Ef(x)=i^m |\det(E)|^{1/2}f(Ex), \qquad E\in \GL(n,\bR),\\
			\label{pQMp}
			&\mathfrak{p}_Qf(x)=\Phi_Q(x)f(x), \qquad Q\in\Sym(n,\bR),
		\end{align}
		$ f\in L^2(\mathbb{R}^{n})$, where $m$ is the argument of $\det(E)^{1/2}$ ({\em Maslov index}) and
		\begin{equation}\label{defPhi}
			\Phi_Q(x)=e^{i\pi Qx\cdot x}
		\end{equation}
		is a chirp. Consequently, metaplectic operators restrict to topological isomorphisms of $\cS(\mathbb{R}^{n})$ and extend to topological isomorphisms of $\cS'(\mathbb{R}^{n})$ by duality:
		\begin{equation}
			\la \hat Sf,g\ra=\la f,\hat S^{-1} g\ra, \qquad f\in\cS'(\rd),\; g\in\cS(\rd).
		\end{equation}
	\end{proposition}
	
	\begin{example}\label{exampleMp}
		\begin{enumerate}[(i)]
			\item The Fourier transform \eqref{FTdef} is a metaplectic operator, and its corresponding projection is $J$, defined as in \eqref{defJ}.
			\item For every $E\in\GL(n,\bR)$ and every $Q\in\Sym(n,\bR)$, we have $\pi^{Mp}(\mathfrak{T}_E)=\cD_E$ and $\pi^{Mp}(\mathfrak{p}_Q)=V_Q$, where the matrices involved are defined in \eqref{defDEVQ}.
			\item The partial Fourier transform with respect to the second $n$ variables of $f\in\cS(\mathbb{R}^{2n})$ is defined as
			\begin{equation}\label{defFT2}	
				\cF_2f(x,\xi)=\int_{\mathbb{R}^{n}}e^{-2\pi i\xi y}f(x,y)dy, \qquad x,\xi\in\mathbb{R}^{n}.
			\end{equation}
			It defines a metaplectic operator on $L^2(\rdd)$ and its projection is the symplectic interchange
			\begin{equation}\label{defAFT2}	
				\cA_{FT2}=\begin{pmatrix}
					I_n & 0_n & 0_n & 0_n\\
					0_n & 0_n & 0_n & I_n\\
					0_n & 0_n & I_n & 0_n\\
					0_n & -I_n & 0_n & 0_n 
				\end{pmatrix}.
			\end{equation}
		\end{enumerate}
	\end{example}
	
	We will use the following two results concerning tensor products of metaplectic operators and the interchange of metaplectic operators and complex conjugation, respectively. 
	
	\begin{proposition}[Appendix B, \cite{cordero2023symplectic}]\label{propGC1}
		Let $\hat S_1,\hat S_2\in\Mp(n,\bR)$. There exists a unique metaplectic operator $\hat S\in\Mp(2n,\bR)$ such that $\hat S(f\otimes g)=\hat S_1f\otimes \hat S_2g$ for every $f,g\in L^2(\rd)$. If
		\begin{equation}
			S_j=\begin{pmatrix}
				A_j & B_j\\
				C_j & D_j
			\end{pmatrix}, \qquad j=1,2
		\end{equation}
		are the projections of $S_1$ and $S_2$, then the projection of $\hat S$ is the symplectic matrix
		\begin{equation}\label{tensmates}
			S=\left(\begin{array}{cc|cc}
				A_1 & 0_n & B_1 & 0_n\\
				0_n & A_2 & 0_n & B_2\\
				\hline
				C_1 & 0_n & D_1 & 0_n\\
				0_n & C_2 & 0_n & D_2\\
			\end{array}\right).
		\end{equation}
		We denote $S=S_1\otimes S_2$ and $\hat S=\hat S_1\otimes \hat S_2$.
	\end{proposition}
	
	\begin{proposition}[Proposition A.2, \cite{cordero2023metaplectic}]\label{propGC2}
		Let $\hat S\in \Mp(n,\bR)$ have projection $S$ with blocks \eqref{blockS}. Then, the operator
		\begin{equation}
			\hat{T}f=\overline{\hat S\bar f}, \qquad f\in L^2(\rd)
		\end{equation}
		is metaplectic, and its projection is the symplectic matrix
		\begin{equation}\label{conjmates}
			T=\begin{pmatrix}
				A & -B\\
				-C & D
			\end{pmatrix}.
		\end{equation}
		We denote $T=\bar S$ and, correspondingly, $\hat T=\hat{\bar S}$.
	\end{proposition}
	Finally, the following covariance property for the Wigner distribution \eqref{defWigner-gg} holds (see \cite{lernermetrics}, p. 61). 
	\begin{proposition}\label{covpropWigner}
	Let $\hat{S}$ be a metaplectic operator with projection $S \in \Sp(n,\bR)$. Then, if $f,g \in L^2(\mathbb{R}^n)$ one has
	\[
	W(\hat{S}f,\hat{S}g)(X)=W(f,g)(S^{-1}(X)), \quad X \in \mathbb{R}^{2n}.
	\]
	\end{proposition}
	
\subsection{Metaplectic Wigner distributions}\cite{ cordero2024wigner,cordero2022wigner}
Our theory is set in the framework of metaplectic Wigner distributions, a class of time-frequency representations generalizing the Wigner distribution using metaplectic operators. They were defined in \cite{cordero2022wigner}, and their properties were later established in \cite{cordero2023symplectic,cordero2024excursus,cordero2024metaplectic,cordero2024wigner,cordero2024unified,cordero2023characterization,cordero2023metaplectic}.
\begin{definition} 
Let $\hat\cA\in \Mp(2n,\bR)$ and let $\cA \in \Sp(2n,\bR)$ be its projection. The corresponding {\em metaplectic Wigner distribution}, or $\cA$-{\em Wigner distribution}, is defined as
\begin{equation}
	W_\cA(f,g)=\hat\cA(f\otimes\bar g), \qquad f,g\in L^2(\mathbb{R}^{n}).
\end{equation}
If $f=g$, we write $W_\cA f=W_\cA(f,f)$.
\end{definition}
\begin{example}\label{exampleMWDs}
	\begin{enumerate}[(i)]
		\item The STFT defined in \eqref{def.STFTintro} is a metaplectic Wigner distribution, indeed 
		 \begin{equation}\label{STFTMp}
		 V_gf=\cF_2\mathfrak{T}_{E_{st}}(f\otimes\bar g), 
		 \end{equation}
		where
		 \begin{equation}
		 	E_{st}=\begin{pmatrix}
				0_n & I_n\\
				-I_n & I_n
			\end{pmatrix},
		 \end{equation}
		 and $\cF_2$ is defined as in Example \ref{exampleMp}. Precisely, $V_gf=\widehat{\cA_{st}}(f\otimes\bar g)$, with
		 \begin{equation}\label{Ast}
			\cA_{st}=\begin{pmatrix}
				I_n & -I_n & 0_n & 0_n\\
				0_n & 0_n & I_n & I_n\\
				0_n & 0_n & 0_n & -I_n\\
				-I_n & 0_n & 0_n & 0_n
			\end{pmatrix},
		\end{equation}
		up to a phase factor.
		 \item The (cross-)$\tau$-Wigner distributions \cite{boggiatto2010time,janssen1985bilinear}, defined by
		 \begin{equation}\label{defWtau}
		 	W_\tau(f,g)(x,\xi)=\int_{\mathbb{R}^{n}}e^{-2\pi i\xi y}f(x+\tau y)\overline{g(x-(1-\tau)y)}dy, \qquad f,g\in L^2(\mathbb{R}^{n}),
		 \end{equation}
		 where $\tau\in\bR$, are metaplectic Wigner distributions. This is due to the fact that 		 \begin{equation}
		 	W_\tau(f,g)=\cF_2\mathfrak{T}_{E_\tau}(f\otimes\bar g), \qquad f,g\in L^2(\rd),
		 \end{equation}
		 where
		  \begin{equation}
		 	E_{\tau}=\begin{pmatrix}
				 I_n & \tau I_n\\
				I_n & -(1-\tau) I_n
			\end{pmatrix}.
		 \end{equation}
		 For $\tau=1/2$ and $\tau=0$ we retrieve the classical (cross-)Wigner distribution \eqref{defWigner-gg} and the Rihackzek distribution \cite{Kirkhood1933,Rihaczek1968}, respectively. Again, $W_\tau(f,g)=\widehat{\cA_\tau}(f\otimes\bar g)$, where
			\begin{equation}\label{Atau}
			\cA_\tau=\begin{pmatrix}
				(1-\tau)I_n & \tau I_n & 0_n & 0_n\\
				0_n & 0_n & \tau I_n & -(1-\tau)I_n\\
				0_n & 0_n & I_n & I_n\\
				-I_n & I_n & 0_n & 0_n
			\end{pmatrix}.
		\end{equation}
		\item An intermediate generalization of the STFT, of the $\tau$-Wigner distributions and of other quadratic time-frequency representations, such as the radar ambiguity function, can be obtained by considering the so-called {\em matrix Wigner distributions}. They are defined as
		\begin{equation}\label{defMatWD}
			W_{\cA_M}(f,g)=W^{(M)}(f,g)=\cF_2\mathfrak{T}_M(f\otimes\bar g), \qquad f,g\in L^2(\rd),
		\end{equation} 
		for a fixed $M\in \GL(2n,\bR)$. If 
		\begin{equation}\label{blockM}
			M=\begin{pmatrix}
				M_{11} & M_{12}\\
				M_{21} & M_{22}
			\end{pmatrix}, \qquad M^{-1}=\begin{pmatrix}
				M_{11}' & M_{12}'\\
				M_{21}' & M_{22}'
			\end{pmatrix},
		\end{equation}
		then the projection of $\hat\cA_M$ is
		\begin{equation}\label{defMtWd}
			\cA_M=\begin{pmatrix}
				M_{11}' & M_{12}' & 0_n & 0_n\\
				0_n & 0_n & M_{12}^\top & M_{22}^\top\\
				0_n & 0_n & M_{11}^\top & M_{21}^\top\\
				-M_{21}' & -M_{22}' & 0_n & 0_n
			\end{pmatrix},
		\end{equation}
		see \cite{cordero2024wigner}, where these distributions are called {\em totally Wigner decomposable}.
		\item The modified Wigner distributions in \cite{PrataDias2018}, arising when considering different symplectic forms on $\rdd$, are rescaled Wigner distributions and, therefore, metaplectic Wigner distributions.
		 \item We also mention that many time-frequency representations of the metaplectic type were considered in the last decades, with applications to signal analysis, cf. \cite{Bai2012,Minh1,Pei2001,ZhangLCWD,ZHANG2023108846,ZhangNewWigner} and the references therein. 
	\end{enumerate}
\end{example}
The continuity properties of Proposition \ref{propContMetap} are naturally transferred to metaplectic Wigner distributions.
\begin{proposition}
	Let $W_\cA$ be a metaplectic Wigner distribution.
	\begin{enumerate}[(i)]
	\item $W_\cA:\cS(\mathbb{R}^{n})\times\cS(\mathbb{R}^{n})\to\cS(\mathbb{R}^{2n})$ is continuous.
	\item $W_\cA:L^2(\mathbb{R}^{n})\times L^2(\mathbb{R}^{n})\to L^2(\mathbb{R}^{2n})$ is bounded and Moyal's identity holds:
	\begin{equation}\label{Moyal}
		\la W_\cA(f_1,g_1),W_\cA(f_2,g_2)\ra = \la f_1,f_2\ra\overline{\la g_1,g_2\ra}, \qquad f_1,g_1,f_2,g_2\in L^2(\rd).
	\end{equation}
	\item $W_\cA:\cS'(\mathbb{R}^{n})\times\cS'(\mathbb{R}^{n})\to\cS'(\mathbb{R}^{2n})$ is continuous.
	\end{enumerate}
\end{proposition}

The many properties of a metaplectic Wigner distribution $W_\cA$ can be explained through the structure of the corresponding projection $\cA\in\Sp(2n,\bR)$. A prototypical example is covariance.

\begin{definition}\label{def29}
	Let $\hat\cA \in \Mp(2n,\mathbb{R})$. The associated Wigner distribution $W_\cA$ is said to be {\em covariant} if for every $x,\xi,y,\eta\in\mathbb{R}^{n}$,
		\begin{equation}
			W_\cA (\pi(x,\xi)f,\pi(x,\xi)g)(y,\eta)=W_\cA(f,g)(y-x,\eta-\xi), \qquad f,g\in \cS(\mathbb{R}^{n}).
		\end{equation}
\end{definition}

Covariant metaplectic Wigner distributions read time-frequency shifts as phase space translations. They can be characterized in terms of their projection 
\begin{equation}\label{blockA}
	\cA=\begin{pmatrix}
		A_{11} & A_{12} & A_{13} & A_{14}\\
		A_{21} & A_{22} & A_{23} & A_{24}\\
		A_{31} & A_{32} & A_{33} & A_{34}\\
		A_{41} & A_{42} & A_{43} & A_{44}
	\end{pmatrix}, \qquad A_{i,j}\in\bR^{n\times n}, \quad i,j=1,\ldots,4
\end{equation}
as follows.

\begin{theorem}\label{thmCovSI}  $W_\cA$ is covariant if and only if 
		\begin{equation}\label{WAcov}
			W_\cA(f,g)=\cF^{-1}\Phi_{-B_\cA}\ast W(f,g), \qquad f,g\in L^2(\mathbb{R}^{n}),
		\end{equation}
		where 
		\begin{equation}\label{defBA}
			B_\cA=\begin{pmatrix}
		A_{13} & \frac{1}{2} I_n-A_{11}\\
		\frac 1 2 I_n-A_{11}^\top & -A_{21}
	\end{pmatrix},
\end{equation}
		and the chirp $\Phi_{-B_\cA}$ is defined as in \eqref{defPhi}.
		
\end{theorem}

We also mention shift-invertible distributions, which play a fundamental role in classifying signals in terms of their time-frequency concentration, see \cite{cordero2023symplectic} and \cite{cordero2024metaplectic} for a more detailed discussion. 

\begin{definition}\label{def292}
Let $\hat \cA \in \Mp(2n,\mathbb{R})$ have projection $\cA$ with $n\times n$ blocks \eqref{blockA}. We say that $W_\cA$ is {\em shift-invertible} if
	\begin{equation}\label{EA}
		E_\cA=\begin{pmatrix}
			A_{11} & A_{13}\\
			A_{21} & A_{23}
			\end{pmatrix}\in\GL(2n,\bR).
	\end{equation}
\end{definition}

\subsection{Pseudodifferential operators}\cite{hormander3,lernermetrics} If $f,g\in\cS(\rd)$, then $W(f,g)\in\cS(\rdd)$. If $a\in\cS'(\rdd)$, we may consider the {\em Weyl-quantized pseudodifferential operator} (or, {\em Weyl operator}) with {\em symbol} $a$, that is the linear operator continuous from $\mathcal{S}(\R^n)$ to $\mathcal{S}'(\R^n)$ defined by
\begin{equation}\label{defOpwSp}
\langle \mathrm{Op}^\w(a)f,g \rangle=\langle a, W(g,f) \rangle, \quad f,g \in \mathcal{S}(\R^n).
\end{equation}

In this work, we focus on symbols without decay, belonging to so called $S_{0,0}^0$-H\"ormander class, for which we recall the definition here. 
\begin{definition}
Let $a\in C^\infty(\mathbb{R}^{2n})$. We say that {\em $a$ is a symbol of order $0$ without decay}, and write $a \in S_{0,0}^0(\mathbb{R}^{2n})$, if for each $\alpha \in \mathbb{N}_0^{2n}$ there exists a constant $C_\alpha>0$ such that 
\[
\abs{\partial_X^\alpha a(X)} \leq C_\alpha, \quad \forall X \in \mathbb{R}^{2n}. 
\]
\end{definition}

If $a\in S^0_{0,0}(\rdd)$, formula \eqref{defOpwSp} can be restated as the oscillatory integral

\begin{equation}
	\Op^{\mathrm{w}}(a)f(x)=\int_{\rdd}e^{2\pi i(x-y)\xi}a\Big(\frac{x+y}{2},\xi\Big)f(y)dyd\xi, \qquad f\in\cS(\rd),
\end{equation}
The Schwartz kernel $k_a$ of $\Op^{\mathrm{w}}(a)$ is related to $a$ by
\begin{equation}\label{kersym}
	a(x,\xi)=\int_{\rd}k_a\left(x+\frac y 2,x-\frac y 2 \right)e^{-2\pi i\xi y}dy.
\end{equation}
We conclude this subsection by recalling two fundamental properties of Weyl operators: their metaplectic invariance and a composition result. We limit here to state these results in the form we need them, and we refer to \cite[Theorem 2.1.2]{lernermetrics} and \cite[Chapter 18]{hormander3}, respectively, for a more detailed discussion.
\begin{theorem}
Let $\hat{S} \in \Mp(n,\mathbb{R})$ with projection $S \in \Sp(n,\mathbb{R})$ and let $a \in \mathcal{S}'(\R^{2n})$. Then
\begin{equation}\label{intertOpwMetap}
	\hat S^{-1}\Op^{\mathrm{w}}(a)\hat S=\Op^{\mathrm{w}}(a\circ S).
\end{equation}
\end{theorem}
\begin{theorem}\label{compweyl} 
Let $a \in S_{0,0}^0(\R^{2n})$ and $b \in S_{0,0}^0(\R^{2n})$. Then, the composition $\mathrm{Op}^{\mathrm{w}}(a)\mathrm{Op}^{\mathrm{w}}(b)$ is well defined and
\[
\mathrm{Op}^{\mathrm{w}}(a)\mathrm{Op}^{\mathrm{w}}(b)=\mathrm{Op}^{\mathrm{w}}(c),
\]
with $c \in S_{0,0}^0(\rdd)$.
\end{theorem}

\subsection{Hamilton map} 
\cite{hormandersymplectic} Let $a$ be a complex quadratic form on the phase space
\[
a:\; \mathbb{R}^{2n} \longrightarrow \mathbb{C}, \quad
X \longmapsto a(X)=\langle X,QX \rangle_{\mathbb{C}^{2n}},
\]
with
$Q \in \bC^{2n\times2n}$ symmetric matrix such that $\mathrm{Re} \, Q \geq 0$. \textit{The Hamilton map} associated with $a=a(X)$ is the (complex) matrix $F$ defined by the equation
\begin{equation}\label{Hamiltonmap}
\sigma(X,FX)=a(X),
\end{equation}
where $\sigma$ is the symplectic form of $\mathbb{R}^{2n}$ extended to $\mathbb{C}^{2n}$ defined by 
\[
\sigma((x,\xi),(y,\eta))=\langle y,\xi \rangle_{\mathbb{C}^n}- \langle x,\eta \rangle_{\mathbb{C}^n}, \quad (x,\xi),(y,\eta)\in \mathbb{C}^{2n},
\]
(where recall $\langle x,y \rangle_{\mathbb{C}^n}=\sum_{j=1}^n x_jy_j$). It is useful to observe that the matrix $F$ can be expressed also as $F=JQ$, where $J$ is given in \eqref{defJ}. Therefore, if $a$ is a \textit{real-valued} quadratic form, one has (with $H_a$ the Hamilton vector field associated with $a$)
\begin{equation}
\label{lin.HVF}
H_{a}(x,\xi)=
\begin{pmatrix}
\partial_{\xi_1} a(x,\xi)\\
\vdots \\
\partial_{\xi_n} a(x,\xi) \\
-\partial_{x_1} a(x,\xi)\\
\vdots \\
-\partial_{x_n} a(x,\xi)
\end{pmatrix}
=2F
\begin{pmatrix}
x \\
\xi
\end{pmatrix}, \quad (x,\xi) \in \mathbb{R}^{2n}.
\end{equation}
and in this case, denoting by $\Phi(t;x,\xi)$ \textit{the bicharacteristic flow}
(that is the flow associated with the Hamilton vector field $H_a$), one has 
\[
\Phi(t;x,\xi)=e^{2tF}\begin{pmatrix}
x \\
\xi
\end{pmatrix}, \quad (x,\xi)\in \mathbb{R}^{2n}, \quad t \in I(x,\xi),
\]
where $I=I(x,\xi) \subseteq \mathbb{R}$ is the interval of definition of the flow starting from the point $(x,\xi) \in \mathbb{R}^{2n}$. 

\section{$\cA$-Wigner kernels}\label{sec:WKers}
The microlocal analysis of pseudodifferential operators with metaplectic Wigner distributions requires the theory of $\cA$-Wigner kernels, extending the already existing Wigner analysis considered in \cite{cordero2024understanding,cordero2023wigner}, limited to the Wigner distribution. Let us briefly synthesize the main results in this particular case. 

For every linear and continuous operator $T:\cS(\rd)\to\cS'(\rd)$ there exists a linear continuous
operator $K:\cS(\rdd)\to\cS'(\rdd)$ so that
\begin{equation}\label{defK}
	W(Tf,Tg)=K(W(f,g)), \qquad f,g\in\cS(\rd).
\end{equation}
The kernel of $K$ is called the {\em Wigner kernel} of $T$, and $K$ acts as an intertwining
operator with the cross-Wigner distribution. In terms of distributions,
the Wigner kernel of $T$ is the (unique) tempered distribution
$k\in\cS'(\bR^{4n})$ such that
\begin{equation}\label{defkWigner}
	\la W(Tf,Tg),W(u,v)\ra = \la k,W(u,v)\otimes \overline{W(f,g)}\ra, \qquad f,g,u,v\in\cS(\rd).
\end{equation}
The Schwartz kernel of $T$, i.e., the distribution $k_T\in\cS'(\rdd)$ such that
\begin{equation}
	\la Tf,g\ra =\la k_T,g\otimes\bar f\ra, \qquad f,g\in\cS(\rd),
\end{equation}
is related to the Wigner kernel by
\begin{equation}\label{WkkT}
	k(x,\xi,y,\eta)=Wk_T(x,y,\xi,-\eta).
\end{equation}
In this section, we study the possibility of performing Wigner analysis with other time-frequency representations of metaplectic type. Our analysis includes the entire spectrum of $\tau$-Wigner distributions, the STFT and, as we shall see, every metaplectic Wigner distribution.
\subsection{Definition of $\cA$-Wigner kernels and main properties}\label{sec:AWig}

Hereafter, $W_\cA$ is a fixed metaplectic Wigner distribution. Loosely speaking, we are interested in a tempered distribution $k_\cA\in\cS'(\bR^{4n})$ with the following kernel property: for every $T:\cS(\rd)\to\cS'(\rd)$ linear and continuous,
\begin{equation}\label{kAintegral}
	W_\cA(Tf,Tg)(X)=\int_{\rdd}k_\cA(X,Y)W_\cA(f,g)(Y)dY, \qquad f,g\in \cS(\rd),
\end{equation}
mimicking the definition of Wigner kernels, with $W_\cA$ replacing the cross-Wigner distribution. We will make use of the following lemma.

\begin{lemma}\label{lemmaDens}
	The $\mbox{span}(\{W_\cA(f,g):f,g\in\cS(\rd)\})$ is dense in $\cS(\rdd)$, whereas $\mbox{span}\{W_\cA(u,v)\otimes \overline{W_\cA(f,g)}:f,g,u,v\in\cS(\rd)\}$ is dense in $\cS(\bR^{4n})$.
\end{lemma}
\begin{proof}
	We use Propositions \ref{propGC1} and \ref{propGC2}. Recall that $\mbox{span}(\{f\otimes\bar g:f,g\in\cS(\rd)\})$ is dense in $\cS(\rdd)$. Since $\hat\cA$ is an isomorphism of $\cS(\rdd)$, we have that $\mbox{span}(\{W_\cA(f,g):f,g\in\cS(\rd)\})$ is dense in $\cS(\rdd)$. The second assertion follows analogously, by writing
	\begin{equation}
		W_\cA(u,v)\otimes \overline{W_\cA(f,g)}=\hat\cA(u\otimes\bar v)\otimes\overline{\hat\cA(f\otimes\bar g)}=\hat\cA(u\otimes\bar v)\otimes\hat{\bar{\cA}}(\bar f\otimes g)=\hat\cB(u\otimes\bar v\otimes \bar f\otimes g),
	\end{equation}
	with $\hat\cB=\hat\cA\otimes\hat{\bar\cA}$, and by using the density of $\mbox{span}(\{u\otimes\bar v\otimes \bar f\otimes g:f,g,u,v\in\cS(\rd)\})$ in $\cS(\bR^{4n})$.
	
\end{proof} 

\begin{theorem}\label{thmkA}
	Let $T:\cS(\rd)\to \cS'(\rd)$ be a linear and continuous operator. There exists a unique tempered distribution $k_\cA\in\cS'(\bR^{4n})$ such that
	\begin{equation}\label{defkA}
		\la W_\cA(Tf,Tg),W_\cA(u,v) \ra = \la k_\cA,W_\cA(u,v)\otimes\overline{W_\cA(f,g)}\ra, \qquad f,g,u,v\in \cS(\rd).
	\end{equation}
	Specifically, if $k_T$ is the Schwartz kernel of $T$, then
	\begin{equation}\label{kAkT}
		k_\cA=(\hat\cA\otimes \hat{\bar\cA}) \mathfrak{T}_{P_+}(k_T\otimes\overline{k_T}),
	\end{equation}
	where $\mathfrak{T}_{P_+}F(x,\xi,y,\eta):=F(x,y,\xi,\eta)$.
\end{theorem}
\begin{proof}
	If $f,g,u,v \in \mathcal{S}(\R^n)$, we have
	\begin{align}
		\la W_\cA(Tf,Tg),W_\cA(u,v)\ra & = \la \hat\cA(Tf\otimes \overline{Tg}), \hat\cA(u\otimes\bar v)\ra \\
		&=\la Tf,u \ra\overline{\la Tg,v\ra}\\
		&=\la k_T,u\otimes\bar f\ra\overline{\la k_T, v\otimes\bar g\ra}\\
		&=\la k_T\otimes\overline{k_T},u\otimes\bar f\otimes \bar v\otimes g\ra\\
		&=\la k_T\otimes\overline{k_T}, \mathfrak{T}_{P_+}(u\otimes\bar v\otimes\bar f\otimes g) \ra\\
		&=\la \mathfrak{T}_{P_+}(k_T\otimes\overline{k_T}), u\otimes\bar v\otimes\bar f\otimes g \ra\\
		&=\la (\hat\cA\otimes \hat{\bar\cA}) \mathfrak{T}_{P_+}(k_T\otimes\overline{k_T}), W_{\cA}(u,v)\otimes\overline{W_\cA(f,g)}\ra.
	\end{align}
	Define $k_\cA=(\hat\cA\otimes \hat{\bar\cA}) \mathfrak{T}_{P_+}(k_T\otimes\overline{k_T})$. To prove the uniqueness, let $k_\cA'\in\cS'(\bR^{4n})$ satisfy \eqref{defkA}. Then, for every $f,g,u,v\in \cS(\rd)$,
	\begin{align}
		\la k_\cA-k_\cA',W_\cA(u,v)\otimes\overline{W_\cA(f,g)}\ra=\la W_\cA(Tf,Tg),W_\cA(u,v) \ra-\la W_\cA(Tf,Tg),W_\cA(u,v) \ra=0.
	\end{align}
	Hence, the assertion follows by Lemma \ref{lemmaDens} using a standard density argument.
	
\end{proof}

\begin{definition}\label{defkA-def}
	The tempered distribution $k_\cA\in\cS'(\bR^{4n})$ of Theorem \ref{thmkA} is the {\em $\cA$-Wigner kernel} of $T$.
\end{definition}
In view of Lemma \ref{lemmaDens}, we see that \eqref{kAintegral} is a restatement of \eqref{defkA} whenever the integral is defined. Hence, we can rephrase Theorem \ref{thmkA} as follows: for every linear and continuous operator $T:\cS(\rd)\to\cS'(\rd)$ there exists a unique linear continuous operator $K_\cA:\cS(\rdd)\to\cS'(\rdd)$ such that
\begin{equation}\label{defKA}
	W_\cA(Tf,Tg)=K_\cA(W_\cA(f,g)),\qquad f,g\in\cS(\rd).
\end{equation}
We call $K_\cA$ the {\em $\cA$-Wigner operator}. By Lemma \ref{lemmaDens}, the $\cA$-Wigner kernel $k_\cA$ of $T$ is the Schwartz kernel of $K_\cA$. 
\subsection{Doubling the variables}
When passing from Schwartz to $\cA$-Wigner kernels, a natural doubling of the number of variables occurs. Recall that the structure of the Wigner distribution allows to write the Wigner kernel $k$ in \eqref{WkkT} as the Wigner transform of the Schwartz kernel, up to a (not so) harmless change of variables. However, a similar expression for the $\cA$-Wigner kernel $k_\cA$ places us in front of an unexpected complication: $W_\cA k_T$ is not defined. While in the particular case of Wigner kernels we can write
\begin{equation}\label{WigKerGood}
	\begin{array}{lcl}
	W_\cA = W & \leadsto & k(x,\xi,y,\eta)=Wk_T(x,y,\xi,-\eta),
	\end{array}
\end{equation}
for general $\cA$-Wigner kernels,
\begin{equation}
	\begin{array}{lcl}
	\text{$W_\cA$ general} & \not\leadsto & k_\cA(x,\xi,y,\eta)=W_\cA k_T(x,y,\xi,-\eta),
	\end{array}
\end{equation}
as $\hat\cA$ is not defined on $L^2(\bR^{4n})$. Thus, on our way to generalize \eqref{WkkT}, by improving \eqref{kAkT}, we are led to face an intermediate challenge: for a given metaplectic operator $\hat\cA$ acting on $L^2(\rdd)$, we have to find $\hat\cA'\in\Mp(4n,\bR)$ so that
\begin{equation}\label{WigAKerGood}
	\begin{array}{lcl}
	\text{$W_\cA$ general} & \leadsto & k_\cA(x,\xi,y,\eta)=W_{\cA'} k_T(x,y,\xi,-\eta).
	\end{array}
\end{equation}
Stated differently, this problem is the same as defining an embedding $\cA\in\Sp(2n,\bR)\hookrightarrow\cA'\in\Sp(4n,\bR)$ properly so that the right side of \eqref{WigAKerGood} holds.
\begin{proposition}
	Let $\cA\in\Sp(2n,\bR)$, consider the permutation matrices
	\begin{equation}\label{originaldefPpm}
		P_\pm=\begin{pmatrix}
			I_n & 0_n & 0_n & 0_n\\
			0_n & 0_n & I_n & 0_n\\
			0_n & I_n & 0_n & 0_n\\
			0_n & 0_n & 0_n & \pm I_n\\
		\end{pmatrix},
	\end{equation}
	and define
	\begin{equation}
		\cA'=\cD_{P_-}(\cA\otimes\bar\cA)\cD_{P_+}\in\Sp(4n,\bR),
	\end{equation}
	(see \eqref{defDEVQ} for the expressions of $\cD_{ P_{\pm}}$). Then,
	\begin{equation}\label{defAprimo}
		\cA'=\left(\begin{array}{cc|cc}
		\diag(A_{11},A_{11}) & \diag(A_{12},A_{12}) & \diag(A_{13},-A_{13}) & \diag(A_{14},-A_{14}) \\
			\diag(A_{21},-A_{21}) & \diag(A_{22},-A_{22}) & \diag(A_{23},A_{23}) &\diag(A_{24},A_{24})\\
			\hline
			\diag(A_{31},-A_{31}) & \diag(A_{32},-A_{32}) & \diag(A_{33},A_{33}) & \diag(A_{34},A_{34})\\
				\diag(A_{41},A_{41})& \diag(A_{42},A_{42}) & \diag(A_{43},-A_{43}) & \diag(A_{44},-A_{44})
		\end{array}\right).
	\end{equation}
\end{proposition}
\begin{proof}
	It follows by an involved matrix multiplication. For the sake of simplicity, let us indicate
	\begin{equation}
		\cA=\left(\begin{array}{cc|cc}
			A_{11} & A_{12} & A_{13} & A_{14}\\
			A_{21} & A_{22} & A_{23} & A_{24}\\
			\hline
			A_{31} & A_{32} & A_{33} & A_{34}\\
			A_{41} & A_{42} & A_{43} & A_{44}
		\end{array}\right)=\begin{pmatrix}
			A & B\\
			C & D
		\end{pmatrix}, \qquad A,B,C,D\in\bR^{2n\times2n}.
	\end{equation}
	Moreover, if
	\begin{equation}
		U=\begin{pmatrix}
			I_n & 0_n\\
			0_n & 0_n
		\end{pmatrix}, \qquad U'=\begin{pmatrix}
			0_n & 0_n\\
			0_n & I_n
		\end{pmatrix}, \qquad R=\begin{pmatrix}
			0_n & I_n\\
			0_n & 0_n
		\end{pmatrix},
	\end{equation}
	then
	\begin{equation}\label{defPpm}
		P_\pm=\begin{pmatrix}
			U & R^\top\\
			R & \pm U'
		\end{pmatrix}.
	\end{equation}
	Consequently,
	\begin{align}
		&\cD_{P_-}(\cA\otimes\bar\cA)\cD_{P_+}\\
		&=
		\left(\begin{array}{cc|cc}
			UAU+R^\top AR & UAR^\top + R^\top AU' & UBU-R^\top BR & UBR^\top-R^\top BU'\\
			RAU-U'AR & RAR^\top -U'AU'& RBU+U'BR & RBR^\top +U'BU'\\
			\hline
			UCU-R^\top CR & UCR^\top-R^\top CU' & UDU+R^\top DR & UDR^\top+R^\top DU'\\
		RCU+U'CR & RCR^\top+U'CU' & RDU-U'DR & RDR^\top-U'DU'
		\end{array}\right)\\
		&=\left(\begin{array}{cc|cc}
		\diag(A_{11},A_{11}) & \diag(A_{12},A_{12}) & \diag(A_{13},-A_{13}) & \diag(A_{14},-A_{14}) \\
			\diag(A_{21},-A_{21}) & \diag(A_{22},-A_{22}) & \diag(A_{23},A_{23}) &\diag(A_{24},A_{24})\\
			\hline
			\diag(A_{31},-A_{31}) & \diag(A_{32},-A_{32}) & \diag(A_{33},A_{33}) & \diag(A_{34},A_{34})\\
				\diag(A_{41},A_{41})& \diag(A_{42},A_{42}) & \diag(A_{43},-A_{43}) & \diag(A_{44},-A_{44})
		\end{array}\right).
	\end{align}
	
\end{proof}

Observe that, in the case of the Wigner distribution, that is when $\cA = \cA_{1/2}$, see \eqref{Atau} with $\tau=1/2$, the diagonal blocks $A_{ij}$ of $\cA_{1/2}$, $i,j=1,\ldots,4$, that in \eqref{defAprimo} appear together with their opposites vanish, whereas the remaining blocks reduce to scalar multiples of the identity. This structural simplification explains why \eqref{WigKerGood} displays the Wigner distribution at both sides of the arrow. Other examples of metaplectic Wigner distributions whose projections enjoy this matrix structure are the STFT, 
\begin{equation}
	k_{st}(x,\xi,y,\eta)=V_{k_T}k_T(x,y,\xi,-\eta),
\end{equation}
where $k_{st}=k_{\cA_{st}}$, see \eqref{Ast}, and the $\tau$-Wigner distributions, for which
\begin{equation}
	k_{\tau}(x,\xi,y,\eta)=W_\tau k_T(x,y,\xi,-\eta),
\end{equation}
where, if $\cA_\tau$ is the projection of $W_\tau$ in \eqref{Atau}, $k_\tau=k_{\cA_{\tau}}$.

\begin{theorem}\label{thmAWigner}
	Let $W_\cA$ be a metaplectic Wigner distribution and $P_\pm$ be defined as in \eqref{originaldefPpm}. Let $\hat\cA'=\mathfrak{T}_{P_-}(\hat\cA\otimes\hat{\bar\cA})\mathfrak{T}_{P_+}$, whose projection is the matrix in \eqref{defAprimo}. Then, for every linear and continuous $T:\cS(\rd)\to\cS'(\rd)$ with Schwartz kernel $k_T$ and $\cA$-Wigner kernel $k_\cA$, it holds:
	\begin{equation}\label{WigAkA}
		k_\cA(x,\xi,y,\eta)=W_{\cA'}k_T(x,y,\xi,-\eta).
	\end{equation}
\end{theorem}
\begin{proof}
	By \eqref{kAkT}, and using that $\mathfrak{T}_{P_-}^2=\id_{\cS'}$, we have:
	\begin{align}
		k_\cA(x,\xi,y,\eta)&=(\hat\cA\otimes\hat{\bar\cA})\mathfrak{T}_{P_+}(k_T\otimes\overline{k_T})(x,\xi,y,\eta)=\mathfrak{T}_{P_-}(\hat\cA\otimes\hat{\bar\cA})\mathfrak{T}_{P_+}(k_T\otimes\overline{k_T})(x,y,\xi,-\eta)\\
		&=W_{\cA'}k_T(x,y,\xi,-\eta).
	\end{align}
	This concludes the proof.
	
\end{proof}

\begin{remark}
	Writing the explicit expression of $W_{\cA'}$ can be quite cumbersome. However, we can always write
	\begin{align}
		W_{\cA'}k_T(x,\xi,y,\eta)=(\mathfrak{T}_{P_+}(\hat\cA\otimes\hat{\bar\cA})\mathfrak{T}_{P_+})(k_T\otimes\overline{k_T})(x,\xi,y,-\eta).
	\end{align}
	The metaplectic operator $\mathfrak{T}_{P_+}(\hat\cA\otimes\hat{\bar\cA})\mathfrak{T}_{P_+}$ acts as a sort of {\em twisted} tensor product, by applying $\hat \cA$ to the first and the third variable, and $\hat{\bar\cA}$ to the second and fourth variable. Hence, up to the sign change due to $\mathfrak{T}_{P_-}$, $W_{\cA'}$ is loosely speaking a tensor product where the variables are shuffled. The rigorous treatment of such variations of metaplectic tensor products requires a generalization of Proposition \ref{propGC1} that is beyond the scope of the present work, and is deferred to future investigation.
\end{remark}

To provide more explicit examples, let us express $W_{\cA'}$ in \eqref{WigAkA} for covariant metaplectic Wigner distributions. Let us recall that $W_\cA$ is covariant if
\begin{equation}
	W_\cA(f,g)=\cF^{-1}\Phi_{-B_\cA}\ast W(f,g),
\end{equation}
where the projection of $W_\cA$ is in the form
\begin{equation}\label{Acov}
	\cA=\begin{pmatrix}
	A_{11} & I_n-A_{11} & A_{13} & A_{13}\\
	A_{21} & -A_{21} & I-A_{11}^\top & -A_{11}^\top\\
	0_d & 0_d & I_n &I_n\\
	-I_n & I_n & 0_n & 0_n
	\end{pmatrix}, \qquad A_{11}\in\bR^{d\times d},\, A_{13},\, A_{21}\in\Sym(n,\bR),
\end{equation}
and $B_\cA$ is as in \eqref{defBA}.

\begin{corollary}\label{corCovAprimo}
	Under the notation above, if $W_\cA$ is covariant, then $W_{\cA'}$ is covariant, with
	\begin{equation}
		W_{\cA'}(f,g)=\cF^{-1}\Phi_{-B_{\cA'}}\ast W(f,g),
	\end{equation}
	with
	\begin{equation}\label{defBAprimo}
		B_{\cA'}=\begin{pmatrix}
		\diag(A_{13},-A_{13}) & \frac{1}{2} I_{2n}-\diag(A_{11},A_{11})\\
		\frac 1 2 I_{2n}-\diag(A_{11},A_{11})^\top & -\diag(A_{21},-A_{21})
	\end{pmatrix}.
	\end{equation}
	In particular, if $T:\cS(\rd)\to\cS'(\rd)$ is a continuous linear operator with Schwartz kernel $k_T$ and $\cA$-Wigner kernel $k_\cA$, then
	\begin{equation}
		k_{\cA}(x,\xi,y,\eta)=(\cF^{-1}\Phi_{-B_{\cA'}}\ast Wk_T)(x,y,\xi,-\eta)
	\end{equation}
	up to a phase factor.
\end{corollary}
\begin{proof}
	It follows by arguing at the level of projections. By writing $\cA'$ when $\cA$ is as in \eqref{Acov} using \eqref{defAprimo}, we obtain
	\begin{equation}
		\cA'=\left(\begin{array}{cc|cc|cc|cc}
			A_{11} & 0_n & I_n-A_{11} & 0_n & A_{13} & 0_n & A_{13} & 0_n\\
			0 & A_{11} & 0_n & I_n-A_{11} & 0_n & -A_{13} & 0_n & -A_{13}\\
			\hline
			A_{21} & 0_n & -A_{21} & 0_n & I_n-A_{11}^\top & 0_n & -A_{11}^\top & 0_n\\
			0_n & -A_{21} & 0_n & A_{21} & 0_n & I_n-A_{11}^\top & 0_n & -A_{11}^\top \\
			\hline
			0_n & 0_n & 0_n & 0_n & I_n & 0_n & I_n & 0_n\\
			0_n & 0_n & 0_n & 0_n & 0_n & I_n & 0_n & I_n \\
			\hline
			-I_n & 0_n & I_n & 0_n & 0_n & 0_n & 0_n & 0_n\\
			0_n & -I_n & 0_n & I_n & 0_n & 0_n & 0_n & 0_n
		\end{array}\right),
	\end{equation}
	which is in the form \eqref{Acov} with $B_{\cA'}$ as in \eqref{defBAprimo}.
\end{proof}

Along the lines of Corollary \ref{corCovAprimo}, we can argue for matrix Wigner distributions, defined in Example \ref{exampleMWDs}. Recall that for $M\in\GL(2n,\bR)$,
\begin{equation}
	W^{(M)}(f,g)(x,\xi)=W_{\cA_M}(f,g)(x,\xi)=\int_{\rd} f(M_{11}x+M_{12}y)\overline{g(M_{21} x + M_{22}y)}e^{-2\pi i\xi y}dy,
\end{equation}
and $\cA_M$ is defined as in \eqref{defMtWd}.

\begin{corollary}
	Under the notation above, if $M$ has blocks as in \eqref{blockM}, then $W_{\cA_M'}$ is the matrix Wigner distribution
	\begin{equation}
		W_{\cA_M'}(f,g)=W^{(M')}(f,g),
	\end{equation}
	up to a phase factor, where
	\begin{equation}\label{defMprimo}
		M'=\begin{pmatrix}
			\diag(M_{11},M_{11}) & \diag(M_{12},M_{12})\\
			\diag(M_{21},M_{21}) & \diag(M_{22},M_{22})
		\end{pmatrix}=M\otimes M.
	\end{equation}
	Consequently, if $T:\cS(\rd)\to\cS'(\rd)$ is a continuous linear operator with Schwartz kernel $k_T$, the corresponding $\cA_M$-Wigner kernel is, up to a phase factor,
	\begin{equation}
		k_{\cA_M'}(x,\xi,y,\eta)=W^{(M')}k_T(x,y,\xi,-\eta).
	\end{equation}
\end{corollary}
\begin{proof}
	By arguing again at the level of symplectic projections, let $M=(M_{ij})_{i,j=1,2}$ and $M^{-1}=(M_{ij}')_{i,j=1,2}$ as in \eqref{blockM}. By \eqref{defAprimo} applied to \eqref{defMtWd}, we get
	\begin{equation}
		\cA_M'=\begin{pmatrix}
			\diag(M_{11}',M_{11}') & \diag(M_{12}',M_{12}') & 0_{2n} & 0_{2n}\\
			0_{2n} & 0_{2n} & \diag(M_{12}^\top,M_{12}^\top) & \diag(M_{22}^\top,M_{22}^\top)\\
			0_{2n} & 0_{2n} & \diag(M_{11}^\top,M_{11}^\top) & \diag(M_{12}^\top,M_{12}^\top)\\
			-\diag(M_{21}',M_{21}') & -\diag(M_{22}',M_{22}') & 0_{2n} & 0_{2n}
		\end{pmatrix},
	\end{equation}
	which is the symplectic matrix in $\Sp(4n,\bR)$ in the form \eqref{defMtWd}, corresponding to the invertible matrix $M'$ defined in \eqref{defMprimo}. This concludes the proof.
	
\end{proof}

\section{$\cA$-Wigner microlocal analysis}\label{sec.Awigmic}

\subsection{The $\cA$-Wigner operator $K_\cA$ and the Weyl quantization}
Metaplectic Wigner distributions enjoy intertwining relations with Weyl operators with symbols in $S^0_{0,0}(\rdd)$, see \cite[Theorem 5.4]{cordero2024wigner}. To be definite, for $a\in\mathcal{S}'(\rdd)$ we consider the symbols
\begin{equation}\label{sigmasigmatilde}\begin{split}
	&\sigma(r,y,\rho,\eta)=a(r,\rho),\\
	&\tilde \sigma(r,y,\rho,\eta)=\bar a(y,-\eta).
	\end{split}
\end{equation}
For the rest of this section, $W_\cA$ is a fixed metaplectic Wigner distribution, and $a\in S^0_{0,0}(\rdd)$ is a fixed symbol.

\begin{proposition}\label{propGCR1}
	Let $\sigma$ and $\tilde{\sigma}$ be the symbols defined as in \eqref{sigmasigmatilde}. Define 
	\begin{align}
		\label{def-b}
		&b(z,w)=(\sigma\circ\cA^{-1})(z,w),\\
		\label{def-tildeb}
		&\tilde b(z,w)=(\tilde \sigma \circ\cA^{-1})(z,w),\\
		\label{def-c}
		&c(z,w)=b(z,w)\tilde b(z,w).
	\end{align}
	Then, $b,\tilde b,c\in S^0_{0,0}(\bR^{4n})$ and, for $f,g\in L^2(\rd)$, the following identities hold in $L^2(\rd)$,
	\begin{align}
		\label{CGR-f1}
		&W_\cA(\Op^\mathrm{w}(a)f,g)=\Op^{\mathrm{w}}(b)W_\cA(f,g),\\
		\label{CGR-f2}
		&W_\cA(f,\Op^\w(a)g)=\Op^\w(\tilde b)W_\cA(f,g),\\
		\label{CGR-f3}
		&W_\cA(\Op^\w(a)f)=\Op^\w(c)W_\cA f.
	\end{align}
\end{proposition}

We will use the following boundedness result for pseudodifferential operators \cite{Toft:2004aa}.
\begin{proposition}\label{propBddL2}
	Let $c\in S^0_{0,0}(\bR^{4n})$. Then, $\Op^\w(c)$ is bounded on $L^2_{v_{s}}(\rdd)$ for every $s\geq0$.
\end{proposition}

In view of Definition \ref{defkA-def}, we need the following improvement of formula \eqref{CGR-f3}, where the metaplectic Wigner distribution $W_\cA f$ is replaced by its polarization $W_\cA(f,g)$.

\begin{proposition}
	Under the assumptions of Proposition \ref{propGCR1},
	\begin{equation}
		W_\cA(\Op^\w(a)f,\Op^\w(a)g)=\Op^\w(c)W_\cA(f,g), \qquad f,g\in L^2(\rd).
	\end{equation}
\end{proposition}
\begin{proof}
	Let $f,g\in L^2(\rd)$, and let $b,\tilde b$ and $c$ be as in \eqref{def-b}, \eqref{def-tildeb} and \eqref{def-c}. By \eqref{CGR-f1} and \eqref{CGR-f2},
	\begin{equation}\begin{split}
		W_\cA(\Op^\w(a)f,\Op^\w(a)g)&=\Op^\w(b)W_\cA(f,\Op^\w(a)g)=\Op^\w(b)\Op^\w(\tilde b)W_\cA(f,g)\\
		&=\Op^\w(c)W_\cA (f,g),
	\end{split}\end{equation}
	and we are done.
	
\end{proof}

Consequently, if $a\in S^0_{0,0}(\rdd)$, then \eqref{defKA} holds for $T=\Op^\w(a)$ with $K_\cA=\Op^\w(c)$, and the $\cA$-Wigner kernel of $\Op^\w(a)$ is the Schwartz kernel of $\Op^\w(c)$, i.e., by \eqref{kersym}
\begin{equation}\label{AWignerKerOpw}
	k_\cA(X,Y)=\int_{\rdd}e^{2\pi i(X-Y)Z}c\Big(\frac{X+Y}{2},Z\Big)dZ.
\end{equation}

\subsection{The $\cA$-Wigner wave front set }
In this section, we propose a definition of Wigner wave fronts for $L^2$ functions. Classically, the definition of Wigner wave front set is related to the notion of the Gabor wave front set, given in Definition \ref{def.WFG}. 
In the literature, Wigner wave fronts have been extended by replacing the STFT with other shift-invertible metaplectic Wigner distributions, such as the cross-$\tau$-Wigner distributions with $0<\tau<1$. The reason for considering shift-invertibility is that shift-invertible distributions are basically rescaled STFT and, therefore, if $W_\cA$ is shift-invertible, and $g$ is a Schwartz window, then $W_\cA(f,g)$ is continuous and Definition \ref{def.WFG} makes sense.

In \cite{cordero2022wigner}, the authors define the so-called $\tau$-{\em Wigner wave front sets} by removing the smoothing window in Definition \ref{def.WFG}, with $W_\tau$ replacing the STFT. The lack of the mitigating effect of $g$ forced the authors to consider $f\in L^2(\rd)$. The {\em $\cA$-Wigner wave front set} is defined as in \eqref{defWFtau-intro} for functions in $L^2(\rd)$. Here, we define the $\cA$-Wigner wave front set and the $\cA$-Wigner wave front set of order $s\geq 0$ as follows. 

\begin{definition}\label{defWFtaus}
Let $f \in L^2(\mathbb{R}^n)$, $\hat\cA \in \Mp(2n,\mathbb{R})$ and $s \geq 0$. 
We say that $0 \neq X_0 \notin WF_\cA^s(f)$, 
\emph{the $\cA$-Wigner wave front set of order $s$} of $f$, 
if there exists an open cone $\Gamma_{X_0}\subseteq \dot{\R}^{2n}$ containing $X_0$ such that $W_\cA f \in L^2_{v_s}(\Gamma_{X_0})$.
\end{definition}

The $\cA$-wave front sets (of order $s$) associated to the $\tau$-Wigner distributions $W_\tau$, are called {\em $\tau$-Wigner wave front set (of order $s$)}, and they are denoted by $WF_\tau(f)$, resp. $WF^s_\tau(f)$. If $\tau=1/2$, we simply write $WF_\mathrm{w}(f)$, the {\em Wigner wave front set of $f$}, similar notation will be used for $WF^s_\mathrm{w}(f)$. 

\begin{remark}
The definition of the $\tau$-Wigner wave front set is a rephrase of the integral definition given in \cite[Definition 5.3]{cordero2022wigner}  for $\cA=\cA_\tau$, see \eqref{Atau}. 
%
\end{remark}

Now, as stated in Theorem 5.4 of \cite{cordero2022wigner}, given  $f \in L^2(\mathbb{R}^n)$ and $\hat \cA \in \Mp(2n,\mathbb{R})$
\[
{WF_\cA (f)}=\emptyset \Rightarrow f \in \mathcal{S}(\mathbb{R}^n).
\]
On the other hand, the $\cA$-Wigner wave front set of order $s \geq 0$ of $f$ roughly encodes the {\em $L^2_{v_s}$-microlocal concentration} of the corresponding $\cA$-Wigner distribution {$W_\cA f$}, and, as expected, this leads to the following proposition. 
\begin{proposition}\label{prop.WFsifffL2vs}
    Let $f \in L^2(\rd)$, $\cA \in \Sp(2n,\mathbb{R})$ and let $s \geq 0$. Then, we have 
    \[
    WF^s_\cA(f)=\emptyset \iff W_\cA f \in L_{v_s}^2(\mathbb{R}^{2n}).
    \]
\end{proposition}
\begin{proof}
    If $W_\cA f \in L_{v_s}^2(\mathbb{R}^{2n})$ we obviously have that $WF^s_\cA(f)=\emptyset$. Conversely, if $WF^s_\cA(f)=\emptyset$, for every $X \in \mathbb{S}^{2n-1}$ there exists an open cone $\Gamma_{X}$ that contains $X$ such that 
    \[
  W_\cA f \in L^2_{v_s}(\Gamma_X).
    \]
    By the compactness of the sphere $\mathbb{S}^{2n-1}$ there exists $X_1,\dots, X_N \in \mathbb{S}^{2n-1}$
    and $\Gamma_{X_1}, \dots, \Gamma_{X_N}$ relative open cones such that
    \begin{equation}\label{coverofcones}
    \dot{\mathbb{R}}^{2n}=\bigcup_{j=1}^N \Gamma_{X_j}
    \end{equation}
    and 
    \begin{equation}\label{L2cone}
 W_\cA f \in L^2_{v_s}(\Gamma_{X_j}), \quad \forall j=1,\dots,n.
    \end{equation}
    Therefore, by choosing a partition of unity $\lbrace \phi_j \rbrace_{j=1}^N$ subordinate to the cover \eqref{coverofcones}, we may write 
    \[
    W_\cA f=\sum_{j=1}^N\phi_j W_\cA f=\sum_{j=1}^N \phi_j\chi_{\Gamma_{X_j}}W_\cA f
    \]
  and the result follows by applying \eqref{L2cone} on every addendum of the right hand-side. 
\end{proof}


\subsection{$\cA$-Wigner microlocality for pseudodifferential operators}\label{sec:MWFs}

This subsection is devoted to proving the so-called \emph{microlocality} of pseudodifferential operators, in terms of the $\cA$-Wigner wave front set, that is Theorem \ref{theo.microlocalitycA}.
For the purpose, we first need the equivalent of \cite[Lemma 5.3]{cordero2023wigner} for $\cA$-Wigner kernels. Its proof requires the following version of Calderón--Vaillancourt Theorem \cite{CalderonVaillancourt}.
\begin{theorem}\label{CVtheorem}
	Let $d\in S^0_{0,0}(\bR^{3n})$. Then, the operator
	\begin{equation}
		Tf(x)=\int_{\rdd}e^{2\pi i(x-y)\zeta}d(x,y,\zeta)f(y)dyd\zeta,
	\end{equation}
	is bounded on $L^2(\rd)$.
\end{theorem}


\begin{lemma}\label{lemma53}
	Let $a\in S^0_{0,0}(\rdd)$ be the symbol of the pseudodifferential operator $\Op^\w(a)$. Let $k_\cA$ be the $\cA$-Wigner kernel of $\Op_\cA(a)$, and $c$ be defined as in \eqref{def-c}. 
\noindent Then, for any integer $N\geq0$,
	\begin{equation}\label{defkN}
		k_N(X,Y):=\la X-Y\ra^{2N} k_\cA(X,Y),
	\end{equation}
	is the kernel of an operator $\Op^\w(c_N)$ with $c_N\in S^0_{0,0}(\rdd)$. Moreover, assume $\phi,\varphi\in\mathcal{C}^\infty(\rd)$ with bounded derivatives of any order and support in two disjoint open cones in $\rd\setminus\{0\}$ for large $z$. Then, for every $s\in \mathbb{R}$ the operator $P_s$ with kernel
	\begin{equation}\label{def-kL}
		\tilde k_s(X,Y)=\phi(X)\la X\ra^s k_\cA(X,Y)\varphi(Y)
	\end{equation}
	is bounded on $L^2(\rd)$.
\end{lemma}
\begin{proof}
	We follow the same pattern of the proof of \cite[Lemma 5.3]{cordero2023wigner}. Let
	\begin{equation}\label{def-cN}
		c_N(X,Z):=\Big(1-\frac{1}{(2\pi)^2}\Delta_Z\Big)^Nc(X,Z).
	\end{equation}
	Clearly, $c_N\in S^0_{0,0}(\bR^{4n})$ because $c\in S^0_{0,0}(\bR^{4n})$. We prove that $k_N$ defined as in \eqref{defkN} is the Schwartz kernel of $\Op^\w(c_N)$. For $f\in\cS(\rdd)$, by using \eqref{AWignerKerOpw}, we have
	\begin{align}
		\int_{\rdd}k_N(X,Y)f(Y)dY&=\int_{\rdd}\la X-Y\ra^{2N} k_\cA(X,Y)f(Y)dY\\
		&=\int_{\rdd}\int_{\rdd}\la X-Y\ra^{2N}e^{2\pi i(X-Y)Z}c\Big(\frac{X+Y}{2},Z\Big)dZ f(Y)dY,
	\end{align}
	which entails that
	\begin{align}
		k_N(X,Y)=\int_{\rdd}\la X-Y\ra^{2N}e^{2\pi i(X-Y)Z}c\Big(\frac{X+Y}{2},Z\Big)dZ.
	\end{align}
	An integration by parts using the classical identity for oscillatory integrals
	\begin{equation}
		\la X-Y\ra^{2N}e^{2\pi i(X-Y)Z}=\Big(1-\frac{1}{4\pi^2}\Delta_Z\Big)^Ne^{2\pi i(X-Y)Z}
	\end{equation}
	(see e.g. \cite[Proposition 0.3]{rodino_book}) yields to
	\begin{equation}\label{kN-duringproof}
		k_N(X,Y)=\int_{\rdd} e^{2\pi i(X-Y)Z}c_N\Big(\frac{X+Y}{2},Z\Big)dZ,
	\end{equation}
	with $c_N$ defined as in \eqref{def-cN}. This concludes the proof of the first part of this Lemma. The second part goes as follows. Let $N\geq0$ to be fixed. By plugging the corresponding \eqref{defkN} into \eqref{def-kL} we find
	\begin{equation}\label{kL-duringproof}
		\tilde k_s(X,Y)=\phi(X)\la X\ra^s\la X-Y\ra^{-2N}k_N(X,Y)\varphi(Y).
	\end{equation}
	The operator with kernel $\tilde k_s$ is 
	\begin{align}
		P_{s}f(X)&=\int_{\bR^{2n}} \tilde k_s(X,Y)f(Y)dY=\int_{\bR^{2n}} \phi(X)\la X\ra^L\la X-Y\ra^{-2N}k_N(X,Y)\varphi(Y)f(Y)dY\\
		&\overset{\eqref{kN-duringproof}}{=}\int_{\bR^{2n}} \phi(X)\la X\ra^s\la X-Y\ra^{-2N}\int_{\rdd} e^{2\pi i(X-Y)Z}c_N\Big(\frac{Z+Y}{2},Z\Big)dZ\varphi(Y)f(Y)dY\\
		&=\int_{\bR^{4n}}e^{2\pi i(X-Y)Z}\underbrace{\phi(X)\la X\ra^s\la X-Y\ra^{-2N}c_N\Big(\frac{X+Y}{2},Z\Big)\varphi(Y)}_{=:d_{s,N}(X,Y,Z)}f(Y)dYdZ.
	\end{align}
	Since $\f$ and $\phi$ are smooth, $d_{s,N}\in \mathcal{C}^\infty(\bR^{3n})$. By the assumption, $\supp \, \phi\cap\supp \, \f\subseteq B_R$ for some $R>0$. Hence, if $X\in\supp \, \phi$ and $Y\in\supp \, \f$,
	\begin{equation}\label{ineqImp}
		\phi(X)\f(Y)\la X\ra\lesssim\chi(X)\f(Y)\la X-Y\ra\la Y\ra\lesssim (1+R^2)^{1/2}\chi(X)\f(Y)\la X-Y\ra.
	\end{equation}
	Consequently, there exists $c_{0,0,0}>0$ such that for every $X,Y,Z\in\rdd$,
	\begin{align}
		|d_{s,N}(X,Y,Z)|&= \Big|\phi(X)\la X\ra^s\la X-Y\ra^{-2N}c_N\Big(\frac{X+Y}{2},Z\Big)\varphi(Y)\Big|\\
		&\lesssim \Big|\phi(X)\la X-Y\ra^{s-2N}c_N\Big(\frac{X+Y}{2},Z\Big)\varphi(Y)\Big|\leq c_{0,0,0}
	\end{align}
	up to choose, say, $N\geq s/2 $. Arguing similarly, we can prove that for every $\alpha,\beta,\gamma\in\bN_0^{2n}$ there exists $c_{\alpha,\beta,\gamma}>0$ such that
	\begin{equation}
	|\partial^\alpha_X\partial^\beta_Y\partial_Z^\gamma d_{s,N}(X,Y,Z)|\leq c_{\alpha,\beta,\gamma}, \qquad X,Y,Z\in\rdd
	\end{equation}
	 i.e., $d_{s,N}\in S^0_{0,0}(\bR^{2n})$. The boundedness of $P_s$ on $L^2(\rd)$ follows then by Theorem \ref{CVtheorem}.

\end{proof}

We are now in the position to give the proof of Theorem \ref{theo.microlocalitycA}, which provide the so-called microlocal property for $\cA$-Wigner wave front sets (once again for any symplectic transformation $\cA$). 

\begin{proof}[Proof of Theorem \ref{theo.microlocalitycA}]
	We use the same technique that proves \cite[Theorem 1.12]{cordero2022wigner}. Let $X_0\notin WF_\cA(f)$. We have to prove that $X_0\notin WF_\cA(\Op^\w(a))$. Precisely, if $X_0\notin WF_\cA(f)$ there exists an open conic neighborhood $\Gamma_{X_0}$ of $X_0$ such that 
	\begin{equation}\label{assump1}
		\int_{\Gamma_{X_0}}\la X\ra^{2s}|W_\cA f(X)|^2dX<\infty, \quad \forall s\geq 0,
	\end{equation}
	and we want to prove that 
	\begin{equation}
		I=\int_{\Gamma_{X_0}'}\la X\ra^{2s}|W_\cA(\Op^\w(a)f)(X)|^2dX<\infty, \quad \forall s\geq0,
	\end{equation}
	for a suitable open conic neighborhood {$\Gamma_{X_0}'$} of $X_0$.
	
	\noindent Thanks to the $\cA$-Wigner analysis of $\Op^\w(a)$ developed so far, we may express $W_\cA(\Op^\w(a)f)$ as
	\begin{equation}
		W_\cA(\Op^\w(a)f)=\Op^\w(c)W_\cA f,
	\end{equation}
	where $c$ is defined as in \eqref{def-c}. Let us consider an open conic neighborhood $\Lambda_{X_0}$ of $X_0$ such that $\Lambda_{X_0}\cap \mathbb{S}^{2n-1}$ is strictly contained in $\Gamma_{X_0}\cap \mathbb{S}^{2n-1}$. Consider a smooth cut-off $\psi\in\mathcal{C}^\infty(\rdd)$, homogeneous of degree 0, such that $\psi|_{\Lambda_{X_0}}\equiv1$, $0\leq\psi\leq1$ and $\supp(\psi)\cap B_R^c\subseteq\Gamma_{X_0}$ for a suitable large $R>0$. Then,
	\begin{equation}
		I\lesssim \underbrace{\int_{\Gamma_{X_0}'}\la X\ra^{2s}|\Op^\w(c)(\psi W_\cA f)(X)|^2dX}_{=:I_1} + \underbrace{\int_{\Gamma_{X_0}'}\la X\ra^{2s}|\Op^\w(c)((1-\psi) W_\cA f)(X)|^2dX}_{=:I_2}.
	\end{equation}
	Thanks to Proposition \ref{propBddL2}, the assumption \eqref{assump1} and the support property of $\psi$, we can estimate
	\begin{equation}
		I_1\leq\norm{\Op^\w(c)(\psi W_\cA f)}_{L^2_{v_s}}\leq \norm{\psi W_\cA f}_{L^2_{v_s}}^2\lesssim \int_{\Gamma_{X_0}}\la X\ra^{2s}|W_\cA f(X)|^2dX<\infty.
	\end{equation}
	The estimate of $I_2$ requires the introduction of a new cut-off function. Let $\Lambda_{X_0}'$ and $\Gamma_{X_0}'$ be open conic neighborhoods of $ X_0 $ with $\Gamma_{X_0}'\subset\subset \Lambda_{X_0}'\subset\subset\Lambda_{X_0}\subset\subset\Gamma_{X_0}$. Let $\phi\in\mathcal{C}^\infty(\rdd)$ be a smooth cut-off, homogeneous of degree 0, with $0\leq\phi\leq1$, $\phi|_{\Gamma_{X_0}'}\equiv1$ and $\supp(\phi)\cap {B_{R'}^c}\subseteq\Lambda_{X_0}'$ for $R'>0$ sufficiently large. {Here, $B_{R'}^c=\rdd\setminus B_{R'}$}. Observe that the assumptions of the second part of Lemma \ref{lemma53} are satisfied by $\phi$ and $\f=1-\psi$. 
	
\noindent By construction of $\phi$, 
	\begin{equation}
		I_2\lesssim \int_{\rdd}|\phi(X)\la X\ra^{s}\Op^\w(c)((1-\psi) W_\cA f)(X)|^2dX.
	\end{equation}
	Now, the operator $P_s$ defined by
	\[
	L^2(\mathbb{R}^{2n}) \ni F \mapsto P_sF:=\phi \, \la \cdot \ra^{s}\Op^\w(c)((1-\psi)F)
	\]
	has kernel in the form \eqref{def-kL}, and therefore, it is bounded on $L^2(\rdd)$ by Lemma \ref{lemma53}. In conclusion,
	\begin{equation}
		I_2\lesssim \norm{P_s(W_\cA)}^2\lesssim\norm{W_\cA f}^2=\norm{f}^4<\infty,
	\end{equation}
	and we are done.
	
\end{proof}

\begin{remark}\label{rmk.stillholdscrosswigner}
Note that the microlocalization arguments used in the proof does not depend on whether we are using the ($\cA$-)Wigner distribution or the cross-($\cA$-)Wigner distribution. In other words, if we replace $W_\cA f$ with $W_\cA(f,g)$ in $I_2$, the same technique still applies.
\end{remark}

\begin{remark}\label{psimicros}
{Let us point out that the microlocal inclusion \eqref{microlocalitycA} holds for any $s \geq 0$,
that is we have the microlocality for symbols in $S_{0,0}^0(\mathbb{R}^n)$ and for the $\cA$-Wigner wave front set at any stratum. }
More precisely, by repeating the proof {of Theorem \ref{theo.microlocalitycA}},
if $a \in S_{0,0}^0(\mathbb{R}^{2n})$ and $f \in L^2(\mathbb{R}^n)$, 
we get that for
each $s \geq 0$
\[
		WF^s_\cA(\Op^\w(a)f)\subseteq WF^s_\cA(f).
\] 
\end{remark}
\section{Wigner microlocality for FIOs with quadratic phase}\label{sec:WMFIOs}
The aim of this section is to prove Theorem \ref{theo.microFIO}, which gives the $\cA$-Wigner microlocality property \eqref{microlocalitycA} for quadratic FIO \cite{cordero2023wigner,cordero2025wigner,cordero2014generalized}, defined as follows. 
\begin{definition}\label{def.quadraticfio}
Let $a \in S_{0,0}^0(\mathbb{R}^{2n})$, let $S \in \Sp(n,\mathbb{R})$ a symplectic map of the form 
\begin{equation}
S= \begin{pmatrix}
A & B \\
C & D
\end{pmatrix}, \quad \mathrm{det} \, A \neq 0,
\end{equation}
where $A,B,C,D \in \bR^{n\times n}$, and let $\Phi(x,\xi)$ 
be the associated quadratic phase, defined by 
\begin{equation}\label{quadphase}
\Phi(x,\xi)=\frac{1}{2} CA^{-1}x \cdot x + A^{-1} x \cdot \xi -\frac{1}{2} A^{-1} B\xi\cdot \xi, \quad (x,\xi) \in \mathbb{R}^{2n}.
\end{equation}
The {\em quadratic FIO} associated with the symplectic map $S$ having phase $\Phi$ and amplitude $a$
is the linear operator 
$\mathcal{K}: \mathcal{S}(\mathbb{R}^n) \rightarrow \mathcal{S}'(\mathbb{R}^n)$
defined by
\begin{equation}\label{FIOI}
\mathcal{K}f(x)= \int_{\mathbb{R}^n} e^{2\pi i \Phi(x,\xi)} a(x,\xi) \hat{f}(\xi)d\xi, \quad f \in \mathcal{S}(\mathbb{R}^n).
\end{equation}
\end{definition}
For the rest of this section, $\cK$ will be a fixed quadratic FIO, associated with $S$ as in \eqref{blockS}, with phase $\Phi$ of the form \eqref{quadphase} and amplitude $a \in S_{0,0}^0(\mathbb{R}^{2n})$. Its Wigner kernel will be denoted by $\tilde\kappa$. We also recall the following representation formula (see \cite{cordero2023wigner}). 
\begin{theorem}\label{theo.K.FIO}
We have:
\begin{equation}\label{K.FIO}
W(\mathcal{K}f,\mathcal{K}g)(x,\xi)=\tilde{\mathcal{K}} W(f,g) (x,\xi)=\int_{\mathbb{R}^{2n}}\tilde{\kappa}(x,\xi,y,\eta)W(f,g)(y,\eta)dyd\eta,
\end{equation}
where the Wigner kernel $\tilde{\kappa}$ is given by 
\[
\tilde{\kappa}(x,\xi,y,\eta)= \int_{\mathbb{R}^{2n}}e^{-2\pi i(z \cdot (\xi-\Phi_x(x,\eta))+r \cdot (y-\Phi_\eta(x,\eta)))}\tilde{a}(x,\eta,z,r) dz dr,
\]
with
\[
\tilde{a}(x,\eta,z,r)=a(x+z/2,\eta+r/2)\overline{a(x-z/2,\eta-r/2)}\in S^0_{0,0}(\mathbb{R}^{2n}).
\]
In addition, for each $N \in \mathbb{N}$, one has 
\[
\begin{split}
\tilde{\kappa}(x,\xi,y,\eta) & =\frac{1}{\langle 2\pi(\xi-\Phi_x(x,\eta),y-\Phi_\eta(x,\eta))\rangle^{2N}} \int_{\mathbb{R}^{2n}}e^{-2\pi i(z \cdot (\xi-\Phi_x(x,\eta))+r \cdot (y-\Phi_\eta(x,\eta)))}\\
&\hspace{6cm} \times (1-\Delta_{z,r})^N\tilde{a}(x,\eta,z,r) dz dr \\
&=\frac{1}{\langle 2\pi(\xi-\Phi_x(x,\eta),y-\Phi_\eta(x,\eta))\rangle^{2N}}h_N(x,\xi,S(y,\eta)),
\end{split}
\]
where $h_N$ is the kernel of a Weyl operator with symbol in $S_{0,0}^0(\mathbb{R}^{4n})$.
\end{theorem}
Therefore, the following auxiliary result holds. 
\begin{lemma}\label{lemma53FIO}
If $\phi,\varphi\in S_{0,0}^0(\rd)$ supported in two disjoint open cones in $\rdd\setminus\{0\}$ for large $X$, we have that for every $s\geq0$ the operator $P_s$ with Schwartz kernel
	\begin{equation}
		\tilde \kappa_s(X,Y)=\phi(X)\la X\ra^s \tilde{\kappa}(X,Y)\varphi(S(Y))
	\end{equation}
	is bounded on $L^2(\rd)$.
\end{lemma}
\begin{proof}
To prove the Lemma we proceed as in Lemma \ref{lemma53},
using the property of the Wigner kernel expressed in Theorem \ref{theo.K.FIO} as follows.

\noindent For every $N \in \mathbb{N}$ we may write 
\[
\tilde \kappa_s(X,Y)=\phi(X)\la X\ra^s\frac{1}{\langle 2\pi(\xi-\Phi_x(x,\eta),y-\Phi_\eta(x,\eta))\rangle^{2N}}h_N(X,S(Y))\varphi(S(Y)),
\]
where $h_N$ is the Schwartz kernel of a pseudodifferential operator
$A_N=\mathrm{Op}^w(a_N)$.
Therefore
\[
\tilde{\kappa}_s(X,Y)= \int e^{i(X-Y)Z}b_{s,N}(X,Y,Z)dZ,
\]
where 
\[
b_{s,N}(X,Y,Z)=\phi(X)\la X\ra^s\frac{1}{\langle 2\pi(\xi-\Phi_x(x,\eta),y-\Phi_\eta(x,\eta))\rangle^{2N}}a_N((X+S(Y))/2,Z)\varphi(S(Y)).
\]
Let us write $S(Y)=(S_1(Y),S_2(Y))$. By applying \cite[Lemma 6.1.4]{cordero2020time} we have
\[
\begin{split}
\abs{b_{s,N}(X,Y,Z)} &\lesssim \abs{\phi(X)\la X\ra^s\frac{1}{\langle (\xi-S_2(Y),x-S_1(Y))\rangle^{2N}}a_N((X+S(Y))/2,Z)\varphi(Y)} \\
&= \abs{\phi(X)\la X\ra^s\frac{1}{\langle X-S(Y)\rangle^{2N}}a_N((X+S(Y))/2,Z)\varphi(S(Y))}.
\end{split}
\]
Hence, by using once more Peetre's inequality and the fact 
that $\mathrm{supp} \, \phi \cap \mathrm{supp} \, \varphi \subseteq B_R$ we get 
\[
\abs{b_{s,N}(X,Y,Z)} \lesssim \abs{\phi(X)\frac{\la (X-S(Y))\ra^s}{\langle (X-S(Y))\rangle^{2N}}a_N((X+S(Y))/2,Z)\varphi(S(Y))},
\]
and so, for $N \in \mathbb{N}$ sufficiently large,
\[
\abs{b_{s,N}(X,Y,Z)} \leq C,
\]
for some constant $C>0$. 
Analogously for every
$\alpha, \beta, \gamma \in \mathbb{N}_0^{2n}$ there exists $c_{\alpha,\beta,\gamma}>0$ such that
	\begin{equation}
	|\partial^\alpha_X\partial^\beta_Y\partial_Z^\gamma b_{s,N}(X,Y,Z)|\leq c_{\alpha,\beta,\gamma}, \qquad X,Y,Z\in\rdd
	\end{equation}
	 i.e., $b_{s,N}\in S^0_{0,0}(\bR^{2n})$. Thus, once again, the boundedness of $P_s$ on $L^2(\rd)$ follows by Theorem \ref{CVtheorem}.
	 
\end{proof}


We are ready to give the proof of Theorem \ref{theo.microFIO}.

\begin{proof}[Proof of Theorem \ref{theo.microFIO}]
The proof follows the same pattern of Theorem \ref{theo.microlocalitycA},
essentially by using the $L^2$ boundedness property of the operator $\tilde{\cK}$ 
of {Theorem \ref{theo.K.FIO}} and
Lemma \ref{lemma53FIO}.

\noindent To prove \eqref{microlocalityFIO} we need to show that
\[
Y_0 \notin WF_\mathrm{w}(f) \Longrightarrow X_0 \notin WF_\mathrm{w}(\mathcal{K}f),
\]
{where $X_0=S(Y_0)$. Assume that $Y_0\notin WF_\mathrm{w}(f)$.} By definition
there exists an open conic neighborhood $\Gamma_{Y_0}$ of $Y_0$ such that 
\begin{equation}\label{assump1FIO}
\int_{\Gamma_{Y_0}}\la Y \ra^{2s}\abs{Wf(Y)}^2dY <+\infty, \quad \forall s\geq 0.
\end{equation}
The goal is to prove that
\begin{equation}
\int_{\Gamma_{X_0}}\la X\ra^{2s}\abs{W(\mathcal{K}f)(X)}^2dX <+\infty, \quad \forall s\geq 0,
\end{equation}
for a suitable open conic neighborhood $\Gamma_{X_0}$ of $X_0$ (to be chosen later).
Recall that by \eqref{K.FIO} we know that 
\[
W(\mathcal{K}f)=\tilde{\mathcal{K}} Wf,
\]
with $\tilde{\mathcal{K}}: L^2(\mathbb{R}^{2n}) \rightarrow L^2(\mathbb{R}^{2n})$
linear bounded operator.
Let us consider an open conic neighborhood $\Lambda_{Y_0} \subseteq \Gamma_{Y_0}$ of $Y_0$ such that
the intersection with the unit sphere $\Lambda_{Y_0}\cap \mathbb{S}^{2n-1}$
is strictly contained in $\Gamma_{Y_0}\cap \mathbb{S}^{2n-1}$. Consider a smooth cut-off $\psi\in\mathcal{C}^\infty(\rdd)$, homogeneous of degree 0, such that $\psi|_{\Lambda_{Y_0}}\equiv1$, $0\leq\psi\leq1$. Then,
	\begin{align}
		 \int_{\Gamma_{X_0}}\la X\ra^{2s}\abs{W(\mathcal{K}f)(X)}^2dX&\lesssim \underbrace{\int_{\Gamma_{X_0}}\la X\ra^{2s}\abs{\tilde{\mathcal{K}}(\psi Wf)(X)}^2dX}_{=:I_1} \\
         &+ \underbrace{\int_{\Gamma_{X_0}}\la X \ra^{2s}\abs{\tilde{\mathcal{K}}((1-\psi) Wf)(X)}^2dX}_{=:I_2}.
	\end{align}
	Hence, since $\tilde{\kappa}(X,Y)=h(X,S(Y))$, where $h$ is the kernel of a pseudodifferential operator with symbol in $S_{0,0}^0(\mathbb{R}^{2n})$ (see \cite{cordero2023wigner}) we get
	\[
	I_1 \leq \| \tilde{\mathcal{K}}(\psi Wf)\|_{L^2_{v_s}} \lesssim \|\psi Wf\|_{L^2_{v_s}}<+\infty.
	\]
The estimate of {$I_2$} requires the introduction of a new cut-off function. 
Let $\Lambda_{X_0}'$ and $\Gamma_{X_0}'$ be open conic neighborhoods of $ X_0 $ with $\Gamma_{X_0}'\subseteq \Lambda_{X_0}'\subseteq S(\Lambda_{Y_0})\subseteq S(\Gamma_{Y_0})$.
Let $\phi\in\mathcal{C}^\infty(\rdd)$ be a smooth cut-off, homogeneous of degree 0,
with $0\leq\phi\leq1$, $\phi|_{\Gamma_{X_0}'}\equiv1$.
Therefore, since the assumptions of Lemma \ref{lemma53FIO} are satisfied
by $\phi$ and $\varphi=(1-\psi)\circ S^{-1}$, by repeating the same argument of Theorem \ref{theo.microlocalitycA} we get that $I_2 <+\infty$.
This concludes the proof of the Theorem.
\end{proof}
\begin{remark}
Note that, as for the pseudodifferential case (see Remark \ref{psimicros}), 
with the notation above, for all $s\in \mathbb{R}$ one has
\[
WF^s_\mathrm{w}(\mathcal{K}f) \subseteq S(WF_\mathrm{w}^s(f)).
\]
\end{remark}

\section{Cross Wigner microlocal analysis}\label{sec.cross}
\subsection{The cross $\cA$-Wigner wave front set}
In order to study \emph{the ghost frequencies}, or more broadly, the possible interaction with two signals,
we now consider the (cross) Wigner Wave front set for $f,g \in L^2(\mathbb{R}^n)$, defined in Definition \ref{def.crossWF}, as a measure of the interaction between $f$ and $g$ in terms of their symplectic correlation $W_\cA(f,g)$. For the rest of this section, $W_\cA$ is a fixed metaplectic Wigner distribution. 

The propagation of the joint phase space concentration of $f$ and $g$ in this context is preserved according to the following result.

\begin{proposition}
Let $a,b \in S_{0,0}^0(\mathbb{R}^{2n})$ and consider $f,g \in L^2(\mathbb{R}^{n})$. Then 
\[
WF_\cA(\Op^\w(a)f,\Op^\w(b)g) \subseteq WF_\cA(f,g).
\]
\end{proposition}
\begin{proof}
By \cite[Theorem 5.4 (ii)]{cordero2024wigner} and Theorem \ref{compweyl} we have
\[
\begin{split}
W_\cA(\Op^\w(a)f,\Op^\w(b)g)&=\Op^\w(\tilde{a})\Op^\w(\tilde{b})W_\cA(f,g)=\Op^\w(c)W_\cA(f,g),
\end{split}
\]
with $\tilde{a},\tilde{b},c \in S_{0,0}^0(\R^{4n})$. Hence, the result follows as in Theorem \ref{theo.microlocalitycA}.
\end{proof}
\subsection{Wigner Microlocal Analysis of Schr\"odinger Interaction}
The goal of this section is to study \textit{the phase space concentration} due
to the interaction of two states that are Schr\"odinger evolutions of two different initial data. 
In the first place, it is fundamental to recall the Sch\"rodinger evolution induced by a \emph{real} quadratic form. Let $a(X)$ be a \textit{real-valued} quadratic form defined by
\[
a(X)=\langle X,QX \rangle_{\R^{2n}}, \quad X \in \R^{2n},
\]  
where $Q \in \bR^{2n\times2n}$ is a symmetric matrix. 
If we consider the evolution induced by $a(X)$ or, more precisely, by $\Op^\w(a)$, with initial datum $u_0 \in L^2(\R^n)$, namely 
\begin{equation}\label{eq.standardschrodinger}
\begin{cases}
i\partial_t u+\Op^\w(a)u=0 & \text{on $\R \times \R^n$},\\
u(0,\cdot)=u_0 & \text{on $\R^n$}
\end{cases}
\end{equation}
we have that the solution of such a problem at time $t\in\bR$ can be expressed as (see \cite{hormandersymplectic}) 
\[
u(t,\cdot)=e^{it\Op^\w(a)}u_0=\widehat{S_t}u_0,
\]
where $S_t \in \Sp(n,\R)$ is given by $S_t=e^{2tF}$, with $F$ the Hamilton map associated with $a$ (see \eqref{Hamiltonmap}).

Let us consider now two different {real} quadratic forms
\[
a_j(X)=\langle X,Q_j X \rangle_{\R^{2n}}, \quad X=(x,\xi) \in \R^{2n}
\]
where $Q_j \in \bR^{2n\times2n}$ is symmetric and let $\sigma_j \in S_{0,0}^0(\mathbb{R}^{2n})$, with $j=1,2$. 
More generally, we may fix $u_{0,1},u_{0,2} \in L^2(\mathbb{R}^n)$, and consider the (quantized) Hamiltonians
\[
H_j=\mathrm{Op}^\w(a_j)+\mathrm{Op}^\w(\sigma_j), \quad j=1,2,
\]
and the corresponding Cauchy problems \eqref{eq.cauchyj}. Their propagators are {\em generalized metaplectic operators}, see \cite{cordero2014generalized,cordero2015integral}, i.e., they can be written as 
\begin{equation}\label{eq.solrepform}
u_j(t,\cdot)=e^{itH_j}u_{0,j}=\widehat{S_{j,t}}\mathrm{Op}^\w(b_{j,t})u_{0,j}, \quad t\in\R,
\end{equation}
for suitable $\widehat{S_{j,t}} \in \Mp(n,\R)$, for $j=1,2$ and $b_{j,t} \in S^{0}_{0,0}(\R^{2n})$. In particular, the metaplectic operators $\widehat{S_{j,t}}$ have projections $S_{j,t}$ that describe the solution of the classical equations of motion with Hamiltonians $a_j(x,\xi)$ in phase-space.

As anticipated in the Introduction, there are two different cases that we may analyze: the case where $a_1=a_2$ and then the case where $a_1\neq a_2$. 
We start with $a_1=a_2$. By writing $S_t=S_{1,t}=S_{2,t}$, we have
\[
u_1(t,\cdot)=\widehat{S_t}\mathrm{Op}^\w(b_{1,t})u_{0,1}, \quad u_2(t,\cdot)=\widehat{S_t}\mathrm{Op}^\w(b_{2,t})u_{0,2}, \quad t\in\bR.
\]
We are ready to prove Theorem \ref{theo.crossa1equala2}.

\begin{proof}[Proof of Theorem \ref{theo.crossa1equala2}]
By abuse, we denote $u_j=u_j(t,\cdot)$, $j=1,2$. Let $X_0 \notin S_t \, WF_\w(u_{0,1},u_{0,2})$, we must prove that $X_0 \notin WF_\w(u_1,u_2)$. To get that, we first note that \eqref{eq.solrepform} entails
\[
W(u_1,u_2)=W(\widehat{S_t}\mathrm{Op}^\w(b_{1,t})u_{0,1},\widehat{S_t}\mathrm{Op}^\w(b_{2,t})u_{0,2}).
\]
Hence, by Proposition \ref{covpropWigner}, we have
\[
W(u_1,u_2)(X)=W(\mathrm{Op}^\w(b_{1,t})u_{0,1},\mathrm{Op}^\w(b_{2,t})u_{0,2})(S_t^{-1}(X)), \quad X \in \bR^{2n}.
\]
Moreover, by \eqref{CGR-f1} and \eqref{CGR-f2}, since $b_{1,t},b_{2,t} \in S_{0,0}^0(\bR^{2n})$, we have that there exist $\tilde{b}_{1,t},\tilde{b}_{2,t} \in S_{0,0}^0(\bR^{4n})$ such that 
\begin{equation}\label{eq.repformpert}
W(u_1,u_2)(X)=\Op^\w(\tilde{b}_{1,t})\Op^\w(\tilde{b}_{2,t})W(u_{0,1},u_{0,2})(S_t^{-1}(X)), \quad X \in \bR^{2n}.
\end{equation}
Thus, by Theorem \ref{compweyl} we have 
\[
W(u_1,u_2)(X)=\Op^\w(d_t)W(u_{0,1},u_{0,2})(S_t^{-1}(X)), \quad X \in \bR^{2n},
\]
where $d_t \in S_{0,0}^0(\R^{4n})$.

Now, assume that $ 0 \neq X_0\notin S_tWF_\w(u_{0,1},u_{0,2})$ and denote by $Y_0=S_t^{-1}(X_0)$. By definition there exists $\Gamma_{Y_0}$ conic neighborhood of $Y_0$ such that 
\[
\int_{\Gamma_{Y_0}}\langle Y \rangle^{2s} \abs{W(u_{0,1},u_{0,2})(Y)}^2dY <+\infty, \quad \forall s \geq 0.
\]
Our aim is to prove that,
\[
\int_{\Gamma_{X_0}}\langle X \rangle^{2s} \abs{W(u_1,u_2)(X)}^2dX <+\infty, \quad \forall s \geq 0,
\]
where $\Gamma_{X_0}$ is the conic neighborhood of $X_0$ defined by $\Gamma_{X_0}:=S_t\Gamma_{Y_0}$.
 
By the representation formula \eqref{eq.repformpert}, by applying the change of variables $X=S_t(Y)$ and using that $\langle Y \rangle \approx \langle S_t(Y) \rangle$ one has 
\[
\begin{split}
\int_{\Gamma_{X_0}}\langle X \rangle^{2s} \abs{W(u_1,u_2)(X)}^2dX &=\int_{\Gamma_{X_0}}\langle X \rangle^{2s} \abs{\Op^\w(d_t)W(u_{0,1},u_{0,2})(S_t^{-1}(X))}^2dX \\
&=\int_{\Gamma_{Y_0}}\langle S_t(Y) \rangle^{2s} \abs{\Op^\w(d_t)W(u_{0,1},u_{0,2})(Y)}^2dY\\
& \lesssim \int_{\Gamma_{Y_0}}\langle Y \rangle^{2s} \abs{\Op^\w(d_t)W(u_{0,1},u_{0,2})(Y)}^2dY, \quad \forall s \geq 0.
\end{split}
\]
The result follows using the same microlocalization argument of Theorems \ref{theo.microlocalitycA} (cf. Remark \ref{rmk.stillholdscrosswigner}).

\end{proof}


We need the following Proposition.
\begin{proposition}\label{prop.ghost}
Let $f,g \in L^2(\mathbb{R}^n)$. Then
\begin{equation}\label{eq.ghost}
WF_{\w}(f+g) \subseteq WF_{\w}(f) \cup WF_{\w}(g) \cup WF_{\w}(f,g).
\end{equation}
\end{proposition}
\begin{proof}
This follows from the fact that
\begin{equation}\label{PolarizationFormula}
W(f+g)=Wf+Wg+2\mathrm{Re} \, W(f,g).
\end{equation}
Indeed, if $0 \neq X_0 \notin WF_\w(f) \cup WF_\w(g) \cup WF_\w(f,g)$ there exist
$\Gamma_1,\Gamma_2, \Gamma_3$ conic neighborhoods of $X_0$ such that  
\[
\chi_{\Gamma_1}Wf, \; \chi_{\Gamma_2}Wg, \; \chi_{\Gamma_3}W(f,g) \in L_{v_s}^2(\mathbb{R}^{2n}), \quad \forall s \geq 0.
\]
Since $X_0\neq0$ is contained in $\Gamma:=\Gamma_1 \cap \Gamma_2 \cap \Gamma_3$, $\Gamma$ is an open cone, and by the polarization formula \eqref{PolarizationFormula}, we obtain that
\[
W(f+g)  \in L^2_{v_s}(\Gamma), \quad \forall s \geq 0.
\]
\end{proof}

The following consequence of Theorem \ref{theo.crossa1equala2} clarifies that when the same initial datum evolves under two distinct Schr\"odinger dynamics, their superposition does not produce any additional energy interaction, the energy distribution simply propagates according to the underlying canonical flow.
\begin{corollary}
Let $u_0 \in L^2(\R^n)$ and $t\in\bR$. Then, under the notation of Theorem \ref{theo.crossa1equala2} with $u_{0,1}=u_{0,2}=u_0$, one has
\[
WF_\w(u_1(t,\cdot)+u_2(t,\cdot)) \subseteq S_t \, WF_\w(u_0).
\]
\end{corollary}
\begin{proof}
The result follows directly by Proposition \ref{prop.ghost} and Theorem \ref{theo.crossa1equala2}.
\end{proof}

More generally, we now prove that the same results hold in the case of 
the $\cA$-Wigner wave front set, if $W_\cA$ is \emph{covariant}.
In this case, recall by Theorem \ref{thmCovSI}, that there exists $\Sigma \in \mathcal{S}'(\R^{2n})$ such that 
\[
W_\cA(f,g)=W(f,g) \ast \Sigma=:Q_\Sigma(f,g), \quad f,g \in \mathcal{S}(\mathbb{R}^n).
\]
Moreover, if we define $\Sigma_t(X):=\Sigma(S_t (X))$, we denote by $W_{\cA_t}=Q_{\Sigma_t}$ the corresponding Wigner distribution (cf. \cite{cordero2023wigner}).
Our generalization follows basically by the following lemma (see \cite[Lemma 7.1]{cordero2024wigner}).
\begin{lemma}\label{lemma.At}
Let $W_\cA$ be covariant. Let $a$ be a real quadratic form and let $S_t=e^{2tF}$, with $F$ the Hamilton map associated with $a$. 
Then 
\[
WF_\cA(\widehat{S_t}f,\widehat{S_t}g)(X)=WF_{\cA_t}(f,g)(S_t^{-1}(X)), \quad f,g \in L^2(\bR^n), \ \ X \in \mathbb{R}^{2n}.
\]
\end{lemma}

\begin{theorem}\label{theo.crossa1equala2cA}
If $\hat\cA \in \Mp(2n,\bR)$ under the notation of Theorem \ref{theo.crossa1equala2} we have 
\[
WF_\cA(u_1(t,\cdot),u_2(t,\cdot)) \subseteq S_t WF_{\cA_t}(u_{0,1},u_{0,2}).
\]
\end{theorem}
\begin{proof}
By Lemma \ref{lemma.At} and Proposition \ref{propGCR1} we have that 
\begin{equation}\label{eq.repformpertcA}
W_\cA(u_1,u_2)(X)=\Op^\w(\tilde{b'}_{1,t})\Op^\w(\tilde{b'}_{2,t})W_{\cA_t}(u_{0,1},u_{0,2})(S_t^{-1}(X)), \quad X \in \bR^{2n}.
\end{equation}
Thus, again by Theorem \ref{compweyl} we have (with $d'_t \in S_{0,0}^0(\R^{4n})$) 
\[
W_\cA(u_1,u_2)(X)=\Op^\w(d'_t)W_{\cA_t}(u_{0,1},u_{0,2})(S_t^{-1}(X)), \quad X \in \bR^{2n}.
\]
Therefore, the result follows as in Theorem \ref{theo.crossa1equala2}.
\end{proof}

Finally, we consider the case of two distinct quadratic forms $a_1 \neq a_2$ and denote by $u_1$ and $u_2$ the corresponding solutions to \eqref{eq.cauchyj}. We prove Theorem \ref{theo.crossa1neqa2}, which essentially says that the phase space concentration of the interaction between $u_1$ and $u_2$ can be expressed in terms of the phase space concentration of the initial datum $u_0$ and a metaplectic operator that depends on the Hamilton maps $F_1$ and $F_2$  associated with $a_1$ and $a_2$, respectively. Note also that this result is stated for general metaplectic Wigner distributions, without the assumptions of covariance. 
\begin{proof}[Proof of Theorem \ref{theo.crossa1neqa2}]
In this case, by Proposition \ref{propGC2} and Proposition \ref{propGC1}, we have
\[
\begin{split}
W_\cA(u_1(t,\cdot),u_2(t,\cdot))&=\hat\cA(u_1(t,\cdot) \otimes \overline{u_2(t,\cdot)})\\
&=\hat\cA(\widehat{S_{1,t}}\mathrm{Op}^\w(b_{1,t})u_{0,1} \otimes \overline{\widehat{S_{2,t}}\mathrm{Op}^\w(b_{2,t})u_{0,2}})\\
&=\hat\cA(\widehat{S_{1,t}}\mathrm{Op}^\w(b_{1,t})u_{0,1} \otimes \widehat{\overline{S_{2,t}}} \mathrm{Op}^\w(b_{2,t}'')\overline{u_{0,2}}) \\
&=\hat\cA\widehat{\cS_t}\mathrm{Op}^\w(b_t)(u_{0,1} \otimes \overline{u_{0,2}}),
\end{split}
\] 
where $\widehat\cB_t=\widehat{S_{1,t}} \otimes \widehat{\overline{S_{2,t}}}$, $b''_{2,t} \in S_{0,0}^0(\R^{2n})$ and $b_t \in S_{0,0}^0(\R^{4n})$.
Now, using the interwining between the Weyl quantization and metaplectic operators \eqref{intertOpwMetap} we get 
\[
W_\cA(u_1,u_2)=\mathrm{Op}^{\mathrm{w}}(b_t \circ (\cA\cS_t)^{-1}) W_{\cC_t}(u_{0,1},u_{0,2}).
\]
Therefore, once again, the result follows as in Theorem \ref{theo.crossa1equala2}.
\end{proof}
\begin{remark}
Once again, all the microlocal inclusions that we proved in the previous theorem hold at any stratum, that is by replacing the $WF_\cA$ with the $WF_\cA^s$, for $s \geq 0$. 
\end{remark}

\begin{remark}
Note that if the initial data $u_{0,1}, u_{0,2} \in \mathcal{S}(\mathbb{R}^n)$, then for all 
$t\in \R$ we have
\[
WF_{\mathcal A}\big(u_1(t,\cdot),u_2(t,\cdot)\big)=\emptyset.
\]
\end{remark}

Recall the definition of shift-invertibility in Definition \ref{def292} and the definition of $E_{\mathcal{A}}$ in \eqref{EA}. By \cite[Proposition 4.4]{cordero2023characterization}, if $W_\cA$ is covariant then $W_{\cA_t}$ originating the wavefront $W_{\cA_t}$ in Theorem \ref{theo.crossa1equala2cA} is also covariant for every $t\in\bR$. Here, we prove that also the shift-invertibility of $W_{\mathcal{A}}$ pass to $\mathcal{C}_t$ for every $t$. This simplifies considerably the computation of $W_{\mathcal{C}_t}$, as exemplified in Example \ref{Exbelow} below.

\begin{theorem}\label{thmGG58}
	Under the notation of Theorem \ref{theo.crossa1neqa2}, if $W_{\cA}$ is a shift-invertible metaplectic Wigner distribution, then $W_{\mathcal{C}_t}$ is shift-invertible for every $t\in\bR$ with $E_{\mathcal{C}_t}=E_{\mathcal{A}}S_{1,t}$.
\end{theorem}
\begin{proof}
	We argue at the level of symplectic projections. Let $\cA=\pi^{Mp}(\hat\cA)$ and $\mathcal{C}_t=S_{1,t}\otimes\overline{S_{2,t}}$, as in Theorem \ref{theo.crossa1neqa2}. Let us consider the permutation $P_+$ defined in \eqref{originaldefPpm}. The right-multiplication of $\cA$ by $P_+$, i.e., $\cA P_+$ swaps the two central columns and rows of $\cA$
	\begin{equation}\label{conjAPpiu}
		\cA P_+=\begin{pmatrix}
			E_{\cA} & \ast\\
			\ast & \ast
		\end{pmatrix},
	\end{equation}
	thereby highlighting the submatrix $E_\cA$ defining shift-invertibility, here the $\ast$ denote the other $2n\times 2n$ blocks of $\cA P_+$, that do not play any role in this proof. 
	Moreover, the conjugation action of $P_+$ on $S_{1,t}\otimes \overline{S_{2,t}}$ gives
	\begin{equation}
		P_+(S_{1,t}\otimes \overline{S_{2,t}})P_+=\diag(S_{1,t},\overline{S_{2,t}}).
	\end{equation}
	Therefore,
	\begin{align}
		\mathcal{C}_t&=\cA\mathcal{S}_t=\cA(S_{1,t}\otimes \overline{S_{2,t}})=\cA P_+\diag(S_{1,t},\overline{S_{2,t}})P_+=\begin{pmatrix}
			E_{\cA} & \ast\\
			\ast & \ast
		\end{pmatrix}
		\begin{pmatrix}
			S_{1,t} & 0_{2d}\\
			0_{2d} & \overline{S_{2,t}}
		\end{pmatrix}P_+\\
		&=\begin{pmatrix}
			E_{\cA}S_{1,t} & \ast\\
			\ast & \ast
		\end{pmatrix}P_+.
	\end{align}
	That is, $E_{\mathcal{C}_t}=E_{\cA}S_{1,t}$. Since $S_{1,t}\in\GL(2n,\bR)$, the assertion follows.
\end{proof}

We conclude this section by applying these results to study the Schrödinger evolution of the concentration properties of the Wigner distribution for the following simple and interesting models.

\begin{example}
In dimension $n=1$, we consider the free particle given by the Schr\"odinger evolution
\begin{equation}\label{eq.freepart}
	\begin{cases}
		i\frac{1}{2\pi}\partial_tu=-\frac{1}{4\pi^2}\Delta u,\\
		u(0,\cdot)=u_0\in L^2(\bR).
	\end{cases}
\end{equation}
In this case, propagator is the metaplectic operators $\widehat{S_{t}}$ with projection
\begin{equation}
	S_{t}=\begin{pmatrix}
		1 & 2t\\
		0 & 1
	\end{pmatrix}.
\end{equation}
Therefore, denoting by $u(t,\cdot)=e^{it\frac{1}{2\pi}\Delta}u_0=\widehat{S_t}u_0$ the solution of \eqref{eq.freepart}, for all $t \in \R$, yields 
\begin{equation}
WF_\w(u(t,\cdot))\subseteq \lbrace (x+2t\xi,\xi); \ (x,\xi) \in WF_\w(u_0)\rbrace.
\end{equation}
\end{example}

\begin{example}
If we consider, the Schr\"odinger evolution induced by the harmonic oscillator (again in $n=1$)
\begin{equation}\label{eq.harosc}
	\begin{cases}
		i\frac{1}{2\pi}\partial_tu=-\frac{1}{8\pi^2}\left(\Delta -x^2\right)u,\\
		u(0,\cdot)=u_0\in L^2(\bR),
	\end{cases}
\end{equation}
since in this case 
\begin{equation}
S_{t}=\begin{pmatrix}
		\cos(t) & \sin(t)\\
		-\sin(t) & \cos(t)
	\end{pmatrix},
\end{equation}
denoting $u(t,\cdot)$ the solution of \eqref{eq.harosc} at time $t \geq 0$, we get that, for all $t \geq 0$,
\begin{equation}
WF_\w(u(t,\cdot))\subseteq \lbrace (\cos(t)x+\sin(t)\xi,-\sin(t)x+\cos(t)\xi); \ (x,\xi) \in WF_\w(u_0)\rbrace.
\end{equation}
\end{example}

\begin{remark}
Let us point out that the inclusions given in the previous two examples, by Theorem \ref{theo.crossa1equala2} still hold for the cross Wigner wave front set of $u_1,u_2$, where $u_1, u_2$ are the solution of the Schr\"odinger equation \eqref{eq.freepart} (or \eqref{eq.harosc}) corresponding to the Hamiltonians $H_j=\frac{1}{4\pi^2}\mathrm{Op}^\w(a)+\mathrm{Op}^\w(\sigma_j)$, with $a=\xi^2$ (or $a=(x^2+\xi^2)/2$) and $\sigma_j \in S^{0}_{0,0}(\R^{2n})$, for $j=1,2$.
\end{remark}

\begin{example}\label{Exbelow}
In dimension $n=1$, we consider the evolution under two Schr\"odinger equations in the form \eqref{eq.standardschrodinger} of two initial data $u_{0,1},u_{0,2}\in L^2(\bR)$. Specifically, we consider the free particle
\begin{equation}
	\begin{cases}
		i\frac{1}{2\pi}\partial_tu_1=-\frac{1}{8\pi^2}\Delta u_1,\\
		u_1(0,x)=u_{0,1}(x),
	\end{cases}
\end{equation}
and the harmonic oscillator
\begin{equation}
	\begin{cases}
		i\frac{1}{2\pi}\partial_tu_2=-\frac{1}{8\pi^2}\left(\Delta -x^2\right)u_2,\\
		u_2(0,x)=u_{0,2}(x).
	\end{cases}
\end{equation}
Again, the propagators for these equations are the metaplectic operators $\widehat{S_{1,t}},\widehat{S_{2,t}}$ with projections
\begin{equation}
	S_{1,t}=\begin{pmatrix}
		1 & 2t\\
		0 & 1
	\end{pmatrix}\qquad \text{and} \qquad 
	S_{2,t}=\begin{pmatrix}
		\cos(t) & \sin(t)\\
		-\sin(t) & \cos(t)
	\end{pmatrix},
\end{equation}
respectively \cite{cordero2025sparse}. Let us consider the Wigner wave front of $u_1$ and $u_2$, whose evolution is governed by the matrix $\mathcal{C}_t$ in Theorem \ref{theo.crossa1neqa2}. Explicitly,
\begin{equation}
	\mathcal{C}_t=\begin{pmatrix}
		1/2 & \cos(t)/2 & 0 & \sin(t)/2\\
		t & \sin(t)/2 & 1/2 & -\cos(t)/2\\
		2t & -\sin(t) & 1 & \cos(t)\\
		-1 & \cos(t) & 0 & \sin(t)
	\end{pmatrix}.
\end{equation}
The metaplectic Wigner distribution associated to $\mathcal{C}_t$ is, up to a phase, given by 
\begin{equation}\label{defCtGG1}
	W_{\mathcal{C}_t}(u_{0,1},u_{0,2})(x,\xi)=W(u_{0,1},\widehat{\delta_{\mathcal{C}_t}}u_{0,2})(x,\xi-2tx),
\end{equation}
where
\begin{equation}
	\widehat{\delta_{\mathcal{C}_t}}u_{0,2}(y)=\begin{cases}
	\frac{1}{|\sin(t)|^{1/2}}e^{-i\pi[2t+\cot(t)]y^2}\int_{-\infty}^\infty e^{i\pi\cot(t)s^2}u_{0,2}(s)e^{2\pi isy/\sin(t)}ds, & \text{if $t\neq 2k\pi$, $k\in\bZ$},\\
	u_{0,2}((-1)^{k+1}y) & \text{if $t=2k\pi$, $k\in\bZ$}.
	\end{cases}
\end{equation}
To compute $W_{\mathcal{C}_t}$ we used that $\mathcal{C}_t$ is shift-invertible for every $t\in\bR$ by Theorem \ref{thmGG58}. Then, \eqref{defCtGG1} follows by \cite[Corollary 4.4]{cordero2024metaplectic} and \cite[Lemma 3.6]{cordero2024unified}, where $W_{\mathcal{C}_t}$ is computed in terms of the short-time Fourier transform and of the Wigner distribution.
\end{example}

\section{Conclusion: a bridge between Wigner and Gabor wave front set}\label{sec:conclusion}
The theory developed in this work heavily relies on the assumption $f\in L^2(\rd)$ in Definition \ref{defWFtau-intro}, because of which we have been able to define $WF_\cA(f)$ for any metaplectic Wigner distribution $W_\cA$. This is due to the fact that $W_\cA f\in \mathcal{C}(\rdd)$ for every $f\in L^2(\rd)$, and it makes sense to consider the integrals
\begin{equation}
	\int_{\Gamma}\la X\ra^{ps}|W_\cA f(X)|^pdX
\end{equation} 
for every $p>0$, $s\in\bR$ and open cone $\Gamma$. On the other hand, the classical microlocal analysis via the Gabor wave front $WF_G(f)$ can be formulated for $f\in\cS'(\rd)$, as it is given in terms of the STFT $V_gf$, for a window $g\in\cS(\rd)\setminus\{0\}$, typically a Gaussian, see Definition \ref{def.WFG}. This definition is well-posed thanks to the smoothing effect of $g$, which guarantees $V_gf\in\mathcal{C}(\rdd)\cap\cS'(\rdd)$. 

In Wigner analysis, we are interested in considering $Wf$, thus avoiding resorting to an additional $g$, both because of the physical meaning of $Wf$ and to get information on the ghost-frequencies, that would be otherwise suppressed by the smoothing effect of the window. However, regardless of ghost frequencies, we may be still interested in obtaining information on the order of growth or decay of $W_\cA f$ (or on the $L^2$ integrability) along cones, also when $f\in\cS'(\rd)$. Reasonably, for fixed $p>0$, we may define $\widetilde{WF_\cA}(f)$ for $f\in\cS'(\rd)$ by replacing condition \eqref{condWFA2} with
\begin{equation}\label{condWFA}
W_\cA f\ast\Phi \in \bigcap_{s \geq 0} L^p_{v_s}(\Gamma_{X_0}),
\end{equation}
where, for instance, $\Phi(X)=2^{n/2}e^{-2\pi|X|^2}=Wg(X)$, being $g(y)=e^{-\pi|y|^2}$. For $p=2$, condition \eqref{condWFA} is clearly not equivalent to \eqref{condWFA2} and the question arises of what kind of analysis we would get with such a correction. The answer is easy to give. In the case of the classical Wigner distribution, i.e., $\widetilde{WF_\cA}=\widetilde{WF_\w}$, the Gaussian filtering $Wf\ast \Phi$ yields the so-called Husimi distribution \cite{Husimi}:
\begin{equation}
	Wf\ast\Phi (X)= |V_gf(X)|^2,
\end{equation}
thereby retrieving $\widetilde{WF_\w}=WF_G$, for $p=1$. In conclusion, the Wigner analysis of wave fronts introduced by Cordero and Rodino intersects Gabor wave front sets when extended to tempered distribution by Gaussian smoothing. This strengthens Wigner analysis as a parallel strategy for the study of evolution operators, propagation of singularities and energy concentration. By working on $Wf$, the gap left by the absence of a window function is filled by the physical interpretability of the Wigner distribution $Wf$ as a quasi-probability distribution, together with its ability to reveal finer microlocal features otherwise suppressed by windowed transforms. At the same time, we do not move far from the classical framework: the familiar Gabor setting and its classical results are recovered through a simple Gaussian convolution.

\section*{Acknowledgements}
The authors thank prof. Luigi Rodino for reading the manuscript and providing significant suggestions and improvements.
The authors are affiliated to the Gruppo Nazionale per Analisi Matematica, la Probabilità e le loro Applicazioni (GNAMPA). Gianluca Giacchi is also supported by the SNSF starting grant ``Multiresolution methods for unstructured data” (TMSGI2 211684).

\bibliographystyle{abbrv}

\end{document}